\documentclass[12pt]{article}
\usepackage{amssymb,amsmath}
\usepackage{amsthm}
\usepackage{cases}
\usepackage{amsfonts}
\usepackage{graphicx}
\usepackage{makecell}
\usepackage{cite,color,xcolor}
\usepackage[margin=1.1 in]{geometry}
\usepackage[colorlinks,citecolor=blue,urlcolor=blue]{hyperref}
\usepackage[utf8]{inputenc}
\usepackage[pagewise]{lineno}

\newtheorem{theorem}{Theorem}[section]
\newtheorem{corollary}{Corollary}[section]
\newtheorem{lemma}{Lemma}[section]
\newtheorem{proposition}{Proposition}[section]

\newtheorem{definition}{Definition}[section]
\newtheorem{remark}{Remark}[section]

\usepackage{txfonts}
\newcommand{\bal}{\begin{align}}
\newcommand{\bbal}{\begin{align*}}
\newcommand{\beq}{\begin{equation}}
\newcommand{\eeq}{\end{equation}}
\newcommand{\bca}{\begin{cases}}
\newcommand{\eca}{\end{cases}}
\def\div{\mathord{{\rm div}}}
\newcommand{\pa}{\partial}
\newcommand{\fr}{\frac}
\newcommand{\na}{\nabla}
\newcommand{\De}{\Delta}

\newcommand{\cd}{\cdot}
\newcommand{\ep}{\varepsilon}
\newcommand{\dd}{\mathrm{d}}

\newcommand{\R}{\mathbb{R}}
\newcommand{\T}{\mathbb{T}}

\newcommand{\f}{\left}
\newcommand{\g}{\right}

\numberwithin{equation}{section}

\begin{document}

\title{Ill-posedness in $B^s_{p,\infty}$ of the Euler equations: Non-continuous dependence}

\author{
 Jinlu Li\footnote{
 School of Mathematics and Computer Sciences,
 Gannan Normal University, Ganzhou 341000, China.
\text{E-mail: lijinlu@gnnu.edu.cn}}
\quad and\quad
Yanghai Yu\footnote{
 School of Mathematics and Statistics,
 Anhui Normal University, Wuhu 241002, China.
\text{E-mail: yuyanghai214@sina.com} (Corresponding author)}
}
\date{\today}
\maketitle
\begin{abstract}
In this paper, we solve an open problem left in the monographs \cite[Bahouri-Chemin-Danchin, (2011)]{BCD}. Precisely speaking, it was obtained in \cite[Theorem 7.1 on pp293, (2011)]{BCD} the existence and uniqueness of $B^s_{p,\infty}$ solution for the Euler equations. We furthermore prove that the solution map of the Euler equation is not continuous in the Besov spaces from $B^s_{p,\infty}$ to $L_T^\infty B^s_{p,\infty}$ for $s>1+d/p$ with $1\leq p\leq \infty$ and in the H\"{o}lder spaces from $C^{k,\alpha}$ to $L_T^\infty C^{k,\alpha}$ with $k\in \mathbb{N}^+$ and $\alpha\in(0,1)$, which later covers particularly the ill-posedness of $C^{1,\alpha}$ solution in \cite[Trans. Amer. Math. Soc., (2018)]{MYtams}. Beyond purely technical aspects on the choice of initial data, a remarkable novelty of the proof is the construction of an approximate solution to the Burgers equation.
\end{abstract}



\section{Introduction}
In this paper, we consider the Cauchy problem of the Euler equations for ideal incompressible fluid
\begin{align}\label{E}
\begin{cases}
\pa_t u+u\cdot \nabla u+\nabla P=0, &\quad (t,x)\in \R^+\times\Omega,\\
\mathrm{div\,} u=0,&\quad (t,x)\in \R^+\times\Omega,\\
u(0,x)=u_0(x), &\quad x\in \Omega,
\end{cases}
\end{align}
where the fluid domain $\Omega=\R^d$ or $\T^d$, the vector field $u(t,x):[0,\infty)\times \Omega\to {\mathbb R}^d$ stands for the velocity of the fluid, the quantity $P(t,x):[0,\infty)\times \Omega\to {\mathbb R}$ denotes the scalar pressure, and $\mathrm{div\,} u=0$ means that the fluid is incompressible.

We say that the Cauchy problem \eqref{E} is Hadamard (locally) well-posed in a Banach space $X$ if for any initial data $u_0\in X$ there exists (at least for a short time) $T>0$ and a unique solution in the space $\mathcal{C}([0,T),X)$ which depends continuously on the initial data. In particular, we say that the solution map is continuous if for any $u_0\in X$, there exists a neighborhood $B \subset X$ of $u_0$ such that for every $v_0 \in B$ the map $v \mapsto V$ from $B$ to $\mathcal{C}([0, T]; X)$ is continuous, where $V$denotes the solution to \eqref{E} with initial data $v_0$. Otherwise the problem is said to be ill-posed. It was pointed out by Kato \cite{Kato} that this notion of well-posedness is rather strong and may not be suitable for certain problems studied in the literature. Instead, it is frequently required that the solution persist in a larger space such as $L^\infty([0, T),X)$ or $C_w([0, T),X)$ (the subscript $w$ indicates weak continuity in the time variable). In this paper, we shall describe situations in which the third well-posedness condition breaks down (for data in Besov spaces).

\subsection{Motivations and Previous Results}

There is by now an extensive literature on the wellposedness theory (especially for existence and uniqueness) for the Euler equations. We refer the interested reader to the monographs of Majda-Bertozzi \cite{MB}, Bahouri-Chemin-Danchin \cite{BCD}, Constantin \cite{P} and the references therein (see e.g. Introduction of Bourgain-Li's series papers \cite{BLgfa,BLim,BLimrn}) for a more comprehensive account. Next, we mainly recall a few progress which are closely related to our problem. The first rigorous local-in-time existence and uniqueness results for the Euler equations were obtained by Lichtenstein \cite{LL} and Gunther \cite{Gu} in H\"{o}lder spaces $C^{k, \alpha}$ with inter $k \geq 1$ and $0<\alpha<1$ (see also a recent breakthrough on finite-time singularity formation for $C^{1, \alpha}$ solution in \cite{E}). Ebin and Marsden \cite{EM70} proved the short time existence, uniqueness, regularity, and continuous dependence on initial conditions for solutions of the Euler equation on general compact manifolds (possibly with $C^{\infty}$ boundary). Their method is to topologize the space of diffeomorphisms by Sobolev $H^s(s>d / 2+1)$-norms and then solve the geodesic equation using contractions. Kato \cite{Karma,Kjfa} proved the local well-posedness of classical solution to the Euler equations in Sobolev spaces $H^s(\mathbb{R}^d)$ for all $s>1+d/2$. Later based on a new type commutator estimate, Kato and Ponce \cite{KPcpam} extended this result to the Sobolev spaces $W^{s, p}(\mathbb{R}^d)$ of the fractional order for $s>d/ p+1$ and $1<p<\infty$. More refined results are available in Besov type spaces which can accommodate $L^{\infty}$ end-point embedding. Vishik \cite{Varma,Vasns} constructed global solutions to 2D Euler in Besov space $B_{p, 1}^{2 / p+1}(\mathbb{R}^2)$ with $1<p<\infty$. For general dimension $d \geq 2$,
Chae \cite{Chae1,Chae2,Chae3} obtained the local well-posedness of the Euler equations in critical Besov space $B_{p, 1}^{d / p+1}(\mathbb{R}^d)$ with $1<p<\infty$ or $s=d/ p+1,1<p<\infty, r=1$. However, this kind of function spaces is only in the $L^p(1<p<\infty)$-framework since the Riesz transform is not bounded on $L^{\infty}$. Pak and Park \cite{Pak1} established existence and uniqueness of solutions of the Euler equations in $B^{1}_{\infty,1}$ and showed that the solution map is in fact Lipschitz
continuous when viewed as a map between $B^{0}_{\infty,1}$ and $\mathcal{C}([0,\infty);B^{0}_{\infty,1})$. Guo, Li and Yin \cite{guo} proved the continuous dependence of the Euler equations in the space $B_{p, r}^s(\mathbb{R}^d)$ with $s>d/ p+1$, $1\leq p\leq \infty, 1 \leq r < \infty$ or $s=d/ p+1,1\leq p\leq \infty, r=1$. A more general well-posedness result can be found in \cite[Theorem 7.1 on pp293]{BCD}. More precisely,
\begin{theorem}[Theorem 7.1, \cite{BCD}]\label{th0}
Let $1 \leq p, r \leq \infty$ and $s \in \mathbb{R}$ be such that $B_{p, r}^s(\mathbb{R}^d) \hookrightarrow C^{0,1}(\mathbb{R}^d)(i.e.\ s>d/ p+1\ \text{or}\ s=d/ p+1,r=1)$. There exists a constant $c$, depending only on $s, p, r$, and $d$, such that for all divergence-free data $u_0 \in B_{p, r}^s(\mathbb{R}^d)$, there exists a time $T \geq c /\left\|u_0\right\|_{B_{p, r}^s(\mathbb{R}^d)}$ such that the Euler equations \eqref{E} has a unique solution $(u, P)$ on $[0, T] \times \mathbb{R}^d$ satisfying particularly that
\bbal
u\in E_{p, r}^s(T):=
\bca
\mathcal{C}\left([0, T] ; B_{p, r}^s(\mathbb{R}^d)\right) \cap \mathcal{C}^1\left([0, T]; B_{p, r}^{s-1}(\mathbb{R}^d)\right), \quad &\mathrm{if} \ r<\infty,\\
\mathcal{C}_w\left(0, T ; B_{p, \infty}^s(\mathbb{R}^d)\right) \cap C^{0,1}(\mathbb{R}^d)\left([0, T] ; B_{p, \infty}^{s-1}(\mathbb{R}^d)\right), \quad &\mathrm{if} \ r=\infty.
\eca
\end{align*}
Namely, if $r<\infty$ (resp., $r=\infty$ ), then $u$ is continuous (resp., weakly continuous) in time with values in $B_{p, r}^s$.
Moreover, there holds
$\|u\|_{L^\infty_TB^s_{p,r}(\mathbb{R}^d)}\leq C\|u_0\|_{B_{p,r}^s(\mathbb{R}^d)}.$
\end{theorem}
From the PDE's point of view, it is crucial to know if an equation which models a physical phenomenon is well-posed in the
Hadamard's sense: existence, uniqueness, and continuous dependence of the solutions with respect to the initial data. In particular, the lack of continuous dependence would cause incorrect solutions or non meaningful solutions. For the Burgers equation
$\partial_t u+u \pa_xu=0,$
which can be the most simple quasilinear symmetric hyperbolic equation, Kato \cite{Kato75} proved the sharpness of the continuous dependence on initial data in the sense of that the solution operator for the Burgers equation in $H^k$ with $k \geq 2$ is continuous but not H\"{o}lder continuous with any prescribed exponent. This example indicates that the continuity of the solution map for a symmetric hyperbolic system is a delicate issue. Systematic studies of ill-posedness of the Cauchy problem \eqref{E} are of a more recent date and concern a wide range of phenomena including gradual loss of regularity
of the solution map. To begin, one can consider explicit solutions to the Euler equations. In studying
measure-valued solutions for 3D Euler equations, DiPerna and Majda \cite{DM} introduced the following shear flow
\begin{align}\label{jql}
u(t, x)=\left(f(x_2), 0, g\left(x_1-t f(x_2)\right)\right), \quad x=\left(x_1, x_2, x_3\right),
\end{align}
where $f$ and $g$ are given single variable functions. This explicit flow (sometimes called $2\frac{1}{2}$-dimensional flow) solves the Euler equations with pressure $P=0$. DiPerna and Lions (see e.g. \cite[pp152]{Lions}) used the above pressureless flow to show that for every $1 \leq p<\infty, T>0, M>0$, there exists a smooth shear-flow for which $\|u(0)\|_{W^{1, p}(\mathbb{T}^3)}=1$ and $\|u(T)\|_{W^{1, p}(\mathbb{T}^3)}>M$. Bardos and Titi \cite{BT} revisited this shear-flow and constructed 3D weak solutions which initially lies in $C^\alpha$ but does not belong to any $C^\beta$ for any $t>0$ and $1>\beta>\alpha^2$ (solutions exhibit instantaneous loss of smoothness in $C^\alpha$ for any $0<\alpha<1$). By this technique, the ill-posedness in the Zygmund space $B_{\infty, \infty}^1$ \cite{BT} and in $\mathrm{LL}_\alpha$ for any $0<\alpha \leq 1$ are obtained by Misiołek and Yoneda \cite{MYcm}.

As pointed out by Kato and Lai \cite{KLjfa}, the question of continuous dependence has long presented a considerable challenge in
the case of the Euler equations and in fact it was established first in \cite{EM70} by non-PDE
techniques.  The continuity of the solution map of the Euler equations in $W^{s,p}$ was obtained (see e.g. \cite{EM70,KLjfa,KPduke} or Appendix in \cite{MYtams}). We would like to mention that the result of Himonas and Misio{\l}ek \cite{HMcmp} who proved that the solution map for the Euler equations in bi(tri)-dimension is not uniformly continuous in Sobolev spaces $H^s(\mathbb{T}^d)$ for $s\in \R$ and in $H^s(\mathbb{R}^d)$ for any $s>0$. Bourgain and Li \cite{BLcmp} settled the border line case $s = 0$.
Cheskidov and Shvydkoy \cite{CS} constructed an initial data in the form
$$u_0(x)=\vec{e}_1\cos x_2+\vec{e}_2\sum\limits^\infty_{j=0}2^{-js}\cos(2^jx_1),$$
 and proved the periodic solutions of the Euler equations are discontinuous (cannot be continuous as a function
of the time variable) in time at $t = 0$ in the Besov spaces $B^s_{p,\infty}(\mathbb{T}^d)$ where $s > 0$ if $2 < p \leq \infty$ and $s>d(2/p-1)$
if $1 \leq p \leq 2$. In particular, it follows that the Cauchy problem \eqref{E} is not well-posed in the sense of Hadamard in $\mathcal{C}([0, T),B^s_{p,\infty}(\mathbb{T}^d)$, although it is known that the corresponding solution map is well defined in $B^s_{p,\infty}(\mathbb{T}^d)$ (see
for instance \cite[Theorem 7.1]{BCD}).
More recently, a big breakthrough is due to Bourgain and Li in their series of papers \cite{BLim,BLgfa,BLimrn}, who proved the strong local ill-posedness (smooth solutions exhibit instantaneous blowup) of the Euler equations in borderline spaces such as $W^{d/p+1,p}$ for any $1 \leq p <\infty$ and $B^{d/p+1}_{p,q}$ for any $1 \leq p <\infty,\ 1 <q \leq \infty$ as well as in the standard spaces $C^k$ and $C^{k-1,1}$ for any integer $k\geq1$ when $d=2,3$ (see also Elgindi and Masmoudi \cite{EM}). The situation is especially intriguing in view of the wellposedness results established by Lichtenstein \cite{LL} and Gunther \cite{Gu} in H\"{o}lder spaces $C^{k, \alpha}$ with inter $k\geq 1$ and $0<\alpha<1$. Subsequently, Misio{\l}ek and Yoneda \cite{MYma} showed that the 2D Euler equations are not locally well-posed in the sense of Hardamard in $C^1$ space and in Besov space $B^1_{\infty,1}$. Based on a DiPerna-Majda type shear flow, Misio{\l}ek and Yoneda \cite{MYtams} showed that the solution map for the Euler equations is not even continuous in the space of H\"{o}lder continuous functions and thus not locally Hadamard well-posed in $C^{1,\alpha}$ with any $\alpha\in(0,1)$ and the continuity of this map is restored in the little H\"{o}lder space $c^{1,\alpha}$.
As pointed out by Misio{\l}ek-Yoneda \cite[Remark 1.4]{MYtams}, their proof is based on a local property of $C^{1,\alpha}(\R^d)$ to construct a counterexample. Also, they mention that it would be interesting to find an explicit counterexample in the Besov space framework $B^{1+\alpha}_{\infty,\infty}$. We would like to emphasize that, the question of continuous dependence of solutions for the Euler equations in more Besov spaces $B_{p, \infty}^s$ with respect to initial conditions has not been explicitly addressed although the existence and uniqueness of $B_{p, \infty}^s$ solutions have been obtained in \cite[Theorem 7.1 on pp293, (2011)]{BCD}. Our main goal in this paper is to revisit the picture of local well-posedness for the Euler equations in more Besov spaces $B_{p, \infty}^s$ with $s>1+d/p$ and $1\leq p\leq\infty$. We shall prove the lack of continuous dependence of the data-to-solution map of the Euler equations which essentially breaks down the well-posedness in the sense of Hadamard.

\subsection{Main results}
From now on, we denote the data-to-solution map of the Euler equations \eqref{E}
\begin{equation*}
\mathbf{S}_t:\begin{cases}
U_R\equiv\f\{u_0\in B_{p,\infty}^s: \|u_0\|_{B^{s}_{p,\infty}}\leq R,\;\mathrm{div\,} u_0=0\g\} \rightarrow L^\infty_T (B_{p, \infty}^{s}),\\
u_0\mapsto \mathbf{S}_t(u_0).
\end{cases}
\end{equation*}
Now, we state our main results as follows.
\begin{theorem}\label{th1}
Let $d\geq 2$. Assume that $(s,p)$ satisfies that
\bal\label{con1}
s>1+\frac dp \quad\text{with}\quad p\in[1,\infty].
\end{align}
The data-to-solution map $
u_0\mapsto \mathbf{S}_t(u_0)
$ of the Euler equations \eqref{E}
is not continuous from any bounded subset in $B^s_{p,\infty}$ into $L^\infty_T(B^s_{p,\infty})$. More precisely, there exists a divergence-free initial data $u_0\in B^{s}_{p,\infty}$ and a perturbation sequence of initial data $u^n_0\in B^{s}_{p,\infty}$ satisfying
\bbal
\|u^n_0-u_0\|_{B^s_{p,\infty}}\to 0,  \; \text{as}\;n\rightarrow \infty,
\end{align*}
which generate the corresponding solutions $\mathbf{S}_t(u_0)$ and $\mathbf{S}_t(u^n_0)$ of \eqref{E} respectively, such that $0<t_n\rightarrow 0$
\bbal
\f\|\mathbf{S}_{t_n}(u^n_0)-\mathbf{S}_{t_n}(u_0)\g\|_{B^s_{p,\infty}}\geq c_0,  \; \text{as}\;n\rightarrow \infty,
\end{align*}
with some positive constant $c_0$ depending on $d,p,s$ but independent of $n$.
\end{theorem}
Due to the fact that $B_{\infty,\infty}^{k+\alpha}=C^{k,\alpha}$ with $k\in \mathbb{N}^+$ and $\alpha\in(0,1)$, we have the following ill-posedness result holds for the Euler equations in H\"{o}lder spaces.
\begin{corollary}\label{co1}
Let $d\geq 2$. Assume that $k\in \mathbb{N}^+$ and $\alpha\in(0,1)$.
The data-to-solution map of the Euler equations \eqref{E} is not continuous from any bounded subset in $C^{k,\alpha}$ into $L^\infty_T(C^{k,\alpha})$.
\end{corollary}
\begin{remark}\label{re2}
Combined the results (Existence and uniqueness) in Lichtenstein \cite{LL} and Gunther \cite{Gu} with Corollary (Non-continuous dependence), we stated that the Euler equations is not locally well-posed (in the sense of Hadamard) in the framework of H\"{o}lder spaces $C^{k,\alpha}$.
\end{remark}
\begin{remark}\label{re3}
As mentioned in \cite{MYtams}, the failure of continuity in Theorem \ref{th1} does not seem to be related to the mechanism described in \cite{BLgfa,EM} which essentially relies on unboundedness of the double Riesz transform in $L^\infty$. It can be explained by the fact that smooth functions are not dense in $B^s_{p,\infty}$. This phenomenon
should be compared with the results of \cite{HMcmp,BLcmp} where it is shown that the solution map cannot be uniformly continuous. However, the ill-posedness mechanism is different and stronger than those of \cite{HMcmp,BLcmp,lyz} since it demonstrates the lack of continuous dependence of the data-to-solution map of the Euler equations which essentially breaks down the well-posedness in the Hadamard's sense.
\end{remark}
As an application of the approach in proving Theorem \ref{th1}, we have
\begin{corollary}\label{co2} Let $d\geq 2$. Assume that $(s,p)$ satisfies \eqref{con1}.
The data-to-solution map $u_0\mapsto \mathbf{S}_t(u_0)$ of the Euler equations \eqref{E} is not continuous as a function of the time variable at time zero in $B^s_{p,\infty}(\R^d)$.
\end{corollary}

\subsection{Difficulties and Strategies to the proof of Theorem \ref{th1}}

We comment on a few points of main difficulties and ideas.
Generally speaking, the continuity of the flow map in a lower norm by the interpolation analysis and in a solution norm by density of smooth
functions is relatively easy to obtain. However, it is unsatisfactory as a few final space, especially for Euler equations in $B^{s}_{p,\infty}(\R^d)$ since the lack of density of smooth functions brings the essential difficulty when proving the continuity with respect to the $B^{s}_{p,\infty}$-norm. For the period space, based on the construction of the explicit solutions of the incompressible Euler equations in the form
\bbal
&u(t,x)=\vec{e}_1+\vec{e}_2\sum\limits^\infty_{j=3}2^{-js}\cos \f(\frac{11}{8}2^{j}(x_1-t)\g),
\end{align*}
we can prove that the periodic solutions of Euler equations that are discontinuous in time at $t = 0$ in $B^s_{p,\infty}(\mathbb{T}^d)$ where $s > 0$ and $1\leq p \leq \infty$. Here we would like to emphasize that, the all aforementioned works (e.g. \cite{BT,MYtams}) involving the shear flow in the form \eqref{jql} do not address the whose space case since it does not belong to any $L^p(\R^d)$ spaces with $1\leq p<\infty$ and thus $B^{s}_{p,\infty}(\R^d)$ with $s>0$. The formulation and the proof of Theorem \ref{th1} for the non-period case requires a deep understanding of how the end-point space $B^{s}_{p,\infty}$-topology changes under the Euler dynamics. In order to introduce the new idea, let us consider the following toy-model:

\noindent{\bf A toy-model: 1-D inviscid Burgers equation.} Let $u_0\in B^s_{p,\infty}$ with $s>1+\fr{1}p$ and $p\in[1,\infty]$.
We consider the Cauchy problem for the one dimensional Burgers equation
\begin{align}\label{B}
\begin{cases}
\pa_t u+u\pa_xu=0,\\
u(0,x)=u_0(x),
\end{cases}
\end{align}
has a unique solution $u(t,x)=u_0(\eta_{u}^{-1}(t,x))$, where the {\it flow map or the particle trajectory map} $\eta_{u}(t,x)$ is induced by the Lipschitz velocity field $u$ and $\eta_{u}^{-1}$ is the inverse of $\eta_{u}$.

\noindent{\bf Step 1:}\;
A basic observation is that, \eqref{B} has a `good' approximation
\begin{align}\label{B1}
\begin{cases}
\pa_t u^{\rm{ap},1}+u_0\pa_xu^{\rm{ap},1}=0,\\
u^{\rm{ap},1}(0,x)=u_0(x).
\end{cases}
\end{align}
Indeed, we have for short time
$u^{\rm{ap},1}-u\rightarrow 0$ in $B^{s-1}_{p,\infty}$.
Compared with \eqref{B}, a nice feature of \eqref{B1} is that it has a unique explicit solution $u^{\rm{ap},1}(t,x)=u_0(\eta_{u_0}^{-1}(t,x))=u_0(x-tu_0(\eta_{u_0}^{-1}(t,x)))$ where $\eta_{u_0}(t,x)=x+tu_0(x)$. Here $\eta_{u_0}^{-1}$ can not be solved explicitly. On the other hand, to put this approximation solution into $B^{s}_{p,\infty}$, we encounter the difficulty that the operator $\Delta_n$ actions on the composition with Lipschitz map. To bypass these, we can furthermore decompose the approximation solution $u^{\rm{ap},1}$. Precisely speaking, due to the linear and local uniqueness of \eqref{B1}, one has the decomposition
$$u^{\rm{ap},1}=\sum_{n=-1}^\infty \Delta_nu^{\rm{ap},1}=\sum_{n=-1}^\infty u_n^{\rm{ap},2},$$
where $u_n^{\rm{ap},2}$ solves
\begin{align}\label{B2}
\begin{cases}
\pa_t u_n^{\rm{ap},2}+u_0\pa_xu_n^{\rm{ap},2}=0,\\
u_n^{\rm{ap},2}(0,x)=\Delta_nu_0(x).
\end{cases}
\end{align}
Based on this, we find that the solution $u_n^{\rm{ap},2}=(\Delta_nu_0)(x-tu_0(\eta_{u_0}^{-1}(t,x)))$ of \eqref{B2} is a `good' approximation of $\Delta_nu^{\rm{ap},1}$.
Since $\eta_{u_0}^{-1}(t,x))$ can be approximated by $x$ for short time, we know that $u_n^{\rm{ap},2}=(\Delta_nu_0)(x-tu_0(\eta_{u_0}^{-1}(t,x)))$ has a `good' approximation $u_n^{\rm{ap},3}=(\Delta_nu_0)(x-tu_0(x))$, which solves
\begin{align*}
\begin{cases}
\pa_t u_n^{\rm{ap},3}+u_0\pa_xu_n^{\rm{ap},3}=-tu_0\pa_xu_0(\pa_x\Delta_nu_0)(x-tu_0(x)), \\
u_n^{\rm{ap},3}(0,x)=\Delta_nu_0(x).
\end{cases}
\end{align*}
In summary, we decompose the solutions of the original Burgers system \eqref{B}
into
\begin{align*}
u=u-u^{\rm{ap},1}+\sum_{n=-1}^\infty (\Delta_nu^{\rm{ap},1}-u^{\rm{ap},2}_n)+\sum_{n=-1}^\infty(u^{\rm{ap},2}_n-u^{\rm{ap},3}_n)+\sum_{n=-1}^\infty u^{\rm{ap},3}_n
\end{align*}
or
\begin{align*}
\Delta_nu=\Delta_n(u-u^{\rm{ap},1})+\Delta_nu^{\rm{ap},1}-u^{\rm{ap},2}_n+u^{\rm{ap},2}_n-u^{\rm{ap},3}_n+u^{\rm{ap},3}_n.
\end{align*}

\noindent{\bf Step 2:}\; Based on {\bf Step 1}, we transform the difference of unknown solutions to the original and perturbed system into that of known approximation solutions. To ensure that the difference between the original and perturbed solution to \eqref{B} in the $B^s_{p,\infty}$-topology is bounded below by a positive constant at any later time, the crucial and skillful step is to construct special original initial data $u_0\in B^s_{p,\infty}$ and perturbed initial data $u_0^n\in B^s_{p,\infty}$ to satisfy that the following three properties
\begin{enumerate}
  \item $\|u_0\|_{B^s_{p,\infty}}\leq C$ and $\|u^{n}_{0}\|_{B^s_{p,\infty}}\leq C$;
  \item $\|u^{n}_{0}-u_0\|_{B^s_{p,\infty}}\to 0$ as $n\to\infty$;
  \item $2^{ns}\|\Delta_n(u_n-u)(t_n)\|_{L^p}\geq \eta_0-C\ep_n$, where $\ep_n\to0$,
\end{enumerate}
We have roughly described the whole strategy of the proof although some technical points could not be mentioned here. For more details see Section \ref{sec3}.

\subsection{Organization of our paper}
In Section \ref{sec2}, we list some notations and known results which will be used in the sequel. In order to expound the main ideas and the construction of initial data, in Section \ref{sec3}, we prove the non-continuous dependence of $B^s_{p,\infty}(\R)$ solution for the 1-D transport type equation. In Section \ref{sec4} (periodic case) and Section \ref{sec5} (non-periodic case), we complete the proof of Theorem \ref{th1}.

\section{Preliminaries}\label{sec2}
{\bf Notation}\; The vector $u(x)=(u_1(x),u_2(x),\ldots,u_d(x))$ with $x=(x_1,x_2,\ldots,x_d)$ and the $i$-th component of the vector $u(x)$ is given by $u^{(i)}(x)$. We write functions depending on time and space as $u(t,x)$ and partial derivatives in time and space are respectively
denoted by $\pa_tu$ and $\pa_{x_i}u$, where $i = 1, \ldots,d$. The metric $\nabla u$ denotes the gradient of $u$ with respect to the $x$ variable, whose $(i,j)$-th component is given by $(\nabla u)_{ij}=\pa_iu_j$ with $1\leq i,j\leq d$.
Throughout this paper, we use letters $C$ (resp. $C_s$) to denote various positive absolute constants (resp. dependent of $s$) whose values may vary from a line to another.
Given a Banach space $X$, we denote its norm by $\|\cdot\|_{X}$. For $I\subset\R$, we denote by $\mathcal{C}(I;X)$ the set of continuous functions on $I$ with values in $X$. Sometimes we will denote $L^p(0,T;X)$ by $L_T^p(X)$. Let us recall that for all $u\in \mathcal{S}$, the Fourier transform $\mathcal{F}u$, also denoted by $\widehat{u}$, is defined by
$$
\mathcal{F}u(\xi)=\widehat{u}(\xi)=\int_{\mathbb{R}^d}e^{-\mathrm{i}x\cd \xi}u(x)\dd x \quad\text{for any}\; \xi\in\mathbb{R}^d.
$$
The inverse Fourier transform allows us to recover $u$ from $\widehat{u}$:
$$
\mathcal{F}^{-1}\widehat{u}(x)=(2\pi)^{-d}\int_{\R^d}e^{\mathrm{i}x\cdot\xi}\widehat{u}(\xi)\dd\xi.$$
We denote the projection
\bbal
&\mathcal{P}: L^{p}(\mathbb{R}^{d}) \rightarrow L_{\sigma}^{p}(\mathbb{R}^{d}) \equiv \overline{\left\{f \in \mathcal{C}^\infty_{0}(\mathbb{R}^{d}) ; {\rm{div}} f=0\right\}}^{\|\cdot\|_{L^{p}(\mathbb{R}^{d})}},\quad p\in(1,\infty),\\
&\mathcal{Q}=\mathrm{Id}-\mathcal{P}.
\end{align*}
In $\mathbb{R}^{d}$, $\mathcal{P}$ can be defined by $\mathcal{P}= \mathrm{Id}+(-\Delta)^{-1}\nabla {\rm{div}}$, or equivalently, $\mathcal{P}=(\mathcal{P}_{i j})_{1 \leqslant i, j \leqslant d}$, where $\mathcal{P}_{i j} \equiv \delta_{i j}+R_{i} R_{j}$ with $\delta_{i j}$ being the Kronecker delta ($\delta_{i j}=0$ for $i\neq j$ and $\delta_{i i}=0$) and $R_{i}$ being the Riesz transform with symbol $-\mathrm{i}\xi_1/|\xi|$. Obviously, $\mathcal{Q}= -(-\Delta)^{-1}\nabla {\rm{div}}$, and if $\div\, u=\div\, v=0$, it holds that
$
\mathcal{Q}(u\cdot\na v)= \mathcal{Q}(v\cdot\na u).
$
Next, we will recall some facts about the Littlewood-Paley decomposition, the nonhomogeneous Besov spaces and their some useful properties (see \cite{BCD} for more details).
Choose a radial, non-negative, smooth function $\vartheta:\R^d\mapsto [0,1]$ such that
 ${\rm{supp}} \,\vartheta\subset B(0, 4/3)$ and $\vartheta(\xi)\equiv1$ for $|\xi|\leq3/4$.
Setting $\varphi(\xi):=\vartheta(\xi/2)-\vartheta(\xi)$, then $\varphi$ satisfies that
 ${\rm{supp}} \;\varphi\subset \left\{\xi\in \mathbb{R}^d: 3/4\leq|\xi|\leq8/3\right\}$ and $\varphi(\xi)\equiv 1$ for $4/3\leq |\xi|\leq 3/2.$
For every $u\in \mathcal{S'}(\mathbb{R}^d)$, the inhomogeneous dyadic blocks ${\Delta}_j$ are defined as follows
\begin{align*}
&\Delta_ju=0,\; \text{if}\; j\leq-2;\quad
\Delta_{-1}u=\vartheta(D)u;\quad\Delta_ju=\varphi(2^{-j}D)u,\;  \text{if}\;j\geq0.
\end{align*}
We recall the definition of the Besov Spaces and norms.
\begin{definition}[\cite{BCD}]
Let $s\in\mathbb{R}$ and $(p,r)\in[1, \infty]^2$. The nonhomogeneous and homogeneous Besov spaces are defined, respectively,
$$
B^{s}_{p,r}(\mathbb{R}^d):=\f\{f\in \mathcal{S}':\;\|f\|_{B^{s}_{p,r}(\mathbb{R}^d)}:=\left\|2^{js}\|\Delta_jf\|_{L^p(\mathbb{R}^d)}\right\|_{\ell^r(j\geq-1)}<\infty\g\}
$$
and
$$
\dot{B}^{s}_{p,r}(\mathbb{R}^d):=\f\{f\in \mathcal{S}'_h:\;\|f\|_{\dot{B}^{s}_{p,r}(\mathbb{R}^d)}:=\left\|2^{js}\|\dot{\Delta}_jf\|_{L^p(\mathbb{R}^d)}\right\|_{\ell^r(j\in \mathbb{Z})}<\infty\g\}.
$$
\end{definition}
\begin{remark} For any $s>0$ and $(p,r)\in[1, \infty]^2$, then $B^{s}_{p,r}=\dot{B}^{s}_{p,r}\cap L^p$ and
$$\|f\|_{B^{s}_{p,r}}\approx \|f\|_{L^{p}}+\|f\|_{\dot{B}^{s}_{p,r}}.$$
In particular, if $s=k+\alpha$ is not an integer, then $B_{\infty, \infty}^s$ is the H\"{o}lder space $C^{k, \alpha}$ with the standard norm
$$
\|f\|_{k, \alpha}=\|f\|_{C^k}+[D^k f]_\alpha
$$
where
$$
[D^k f]_\alpha=\sum_{|\beta|=k} \sup _{x \neq y} \frac{\left|D^\beta f(x)-D^\beta f(y)\right|}{|x-y|^\alpha}, \quad 0<\alpha<1, k \in \mathbb{N}.
$$
\end{remark}
Next we recall the following product law which will be used often throughout the paper.
\begin{lemma}[\cite{BCD}]\label{lp}
Assume $(s,p)$ satisfies \eqref{con1} and $\sigma>0$. Then
 there exists a constant $C$, depending only on $d,p,r,\sigma$ or $s$ such that
$$\|fg\|_{B^{\sigma}_{p,\infty}(\mathbb{R}^d)}\leq C\f(\|f\|_{L^\infty(\mathbb{R}^d)}\|g\|_{B^\sigma_{p,\infty}(\mathbb{R}^d)}
+\|g\|_{L^\infty(\mathbb{R}^d)}\|f\|_{B^{\sigma}_{p,\infty}(\mathbb{R}^d)}\g),\quad \forall f,g\in L^\infty\cap B^{\sigma}_{p,\infty}(\mathbb{R}^d).$$
Furthermore, due to the embedding $B^{s-1}_{p,\infty}\hookrightarrow L^{\infty}$, there holds for any $(f,g)\in B^{s-1}_{p,r}\times B^s_{p,\infty}(\mathbb{R}^d)$
$$\|f\cdot \na g\|_{B^{s-1}_{p,\infty}(\mathbb{R}^d)}\leq C\|f\|_{B^{s-1}_{p,r\infty}(\mathbb{R}^d)}\|g\|_{B^s_{p,\infty}(\mathbb{R}^d)},$$
and for any $(f,g)\in B^{s}_{p,\infty}\times B^{s+1}_{p,\infty}(\mathbb{R}^d)$
\begin{align*}
\|f\cdot \na g\|_{B^{s}_{p,\infty}(\mathbb{R}^d)}\leq C\f(\|f\|_{B^{s-1}_{p,\infty}(\mathbb{R}^d)}\|g\|_{B^{s+1}_{p,\infty}(\mathbb{R}^d)}+\|f\|_{B^{s}_{p,\infty}(\mathbb{R}^d)}\|g\|_{B^{s}_{p,\infty}(\mathbb{R}^d)}\g).
\end{align*}
\end{lemma}

\begin{lemma}[\cite{guo}]\label{lem:P}
Assume that $(s,p)$ satisfies \eqref{con1}. Then there exists a constant $C$, depending only on $d,p,r,s$, such that for any $u,v\in B^s_{p,\infty}(\mathbb{R}^d)$ with $\mathrm{div\,} u=\mathrm{div\,} v=0$
\begin{align*}
&\|\mathcal{Q}(u\cdot \na v)\|_{B^s_{p,\infty}(\mathbb{R}^d)}\leq C \|u\|_{B^s_{p,r}(\mathbb{R}^d)}\|v\|_{B^s_{p,\infty}(\mathbb{R}^d)},\\
&\|\mathcal{Q}(u\cdot \na v)\|_{B^{s-1}_{p,\infty}(\mathbb{R}^d)}\leq C \min\f\{\|u\|_{B^{s-1}_{p,\infty}(\mathbb{R}^d)}\|v\|_{B^s_{p,\infty}(\mathbb{R}^d)},\, \|v\|_{B^{s-1}_{p,\infty}(\mathbb{R}^d)}\|u\|_{B^s_{p,\infty}(\mathbb{R}^d)}\g\}.
\end{align*}
\end{lemma}
We recall the commutator estimation which will be used in the sequel.
\begin{lemma}[Commutator estimation, \cite{BCD}]\label{jhz}
For $1\leq p\leq\infty$ and $s>0$, there exists a constant $C=C(p,s)>0$ such that
 \begin{align*}
&\sup_{k\geq-1}\f(2^{ks}\|[\Delta_{k},v]\cdot\nabla f\|_{L^{p}(\mathbb{R}^d)}\g)\leq C\f(\|\nabla v\|_{L^{\infty}(\mathbb{R}^d)}\|f\|_{{B}_{p,\infty}^{s}(\mathbb{R}^d)}+\|\nabla f\|_{L^{\infty}(\mathbb{R}^d)}\|\nabla v\|_{{ B}_{p,\infty}^{s-1}(\mathbb{R}^d)}\g),
\end{align*}
where we denote the standard commutator $[\Delta_k,v]\cdot\nabla f=\Delta_k(v\cdot\nabla f)-v\cdot\Delta_{k}\nabla f$.
\end{lemma}

Finally, we recall the regularity estimate for the transport equations.
\begin{lemma}[\cite{BCD}]\label{zzgj}
Let $1\leq p,r\leq \infty$. Assume that
\begin{align*}
\sigma> -d \min\f(\frac{1}{p}, \frac{1}{p'}\g) \quad \mathrm{or}\quad \sigma> -1-d \min\f(\frac{1}{p}, \frac{1}{p'}\g)\quad \mathrm{if} \quad \mathrm{div\,} v=0.
\end{align*}
There exists a constant $C$, depending only on $d,p,r,\sigma$, such that for any smooth solution $f$ of the following transport equations
\begin{align*}
\begin{cases}
\pa_t f+v\cdot \nabla f=g\in L_{\rm{loc }}^1\left(\mathbb{R}^{+}; B_{p, r}^\sigma(\mathbb{R}^d)\right), \\
f(0,x)=f_0(x)\in B^\sigma_{p,r}(\mathbb{R}^d),
\end{cases}
\end{align*}
we have for $t\geq 0$
\begin{align}\label{ES2}
\|f\|_{L^\infty_tB^{\sigma}_{p,r}(\mathbb{R}^d)}\leq Ce^{CV_{p}(v,t)}\f(\|f_0\|_{B^\sigma_{p,r}(\mathbb{R}^d)}+\int^t_0\|g(\tau)\|_{B^\sigma_{p,r}(\mathbb{R}^d)}\dd \tau\g),
\end{align}
with
\begin{align*}
V_{p}(v,t)=
\begin{cases}
\int_0^t \|\nabla v(s)\|_{B^{\frac{d}{p}}_{p,\infty}(\mathbb{R}^d)\cap L^\infty(\mathbb{R}^d)}\dd s,\quad \mathrm{if} \quad \sigma<1+\frac{d}{p},\\
\int_0^t \|\nabla v(s)\|_{B^{\sigma-1}_{p,r}(\mathbb{R}^d)}\dd s, \quad \quad \mathrm{if} \quad \sigma>1+\frac{d}{p}\ \mathrm{or}\ \{\sigma=1+\frac{d}{p} \mbox{ and } r=1\}.
\end{cases}
\end{align*}
If $f=v$, then for all $\sigma>0$ ($\sigma>-1$, if $\mathrm{div\,} v=0$), the estimate \eqref{ES2} holds with
$V_{p}(t)=\int_0^t \|\nabla v(s)\|_{L^\infty(\mathbb{R}^d)}\dd s.$
\end{lemma}

\section{Non-Continuous Dependence of 1-D Transport Equation}\label{sec3}
In order to expound the main ideas without clouding it by technicalities, the purpose of this section is treat the 1-D transport type equation. Particularly, we prove the non-continuous dependence of $B^s_{p,\infty}(\R)$ solution for the 1-D transport type equation:

\begin{equation}\label{TR}
\begin{cases}
u_t+u\pa_xu=F(u), \qquad &(t,x)\in \R^+\times\R,\\
u(t=0,x)=u_0(x), \qquad &x\in \R.
\end{cases}
\end{equation}
We can establish the following
\begin{theorem}\label{th3}
Let $u_0\in B^s_{p,\infty}(\mathbb{R})$ with $s>1+1/p$ and $p\in[1,\infty]$. Assume that $F(u)$ satisfies that for any $t\geq0$
\begin{itemize}
  \item Uniform bounded in $B^s_{p,\infty}(\R)$:\;
  $\|F(u)\|_{B^s_{p,\infty}(\R)}\leq C\|u\|_{B^s_{p,\infty}(\R)};$
  \item Lipchitz continuous in $B^{s-1}_{p,\infty}(\R)$:\;
  $\|F(u)-F(v)\|_{B^{s-1}_{p,\infty}(\R)}\leq C\|u-v\|_{B^{s-1}_{p,\infty}(\R)}.$
\end{itemize}
Then there exists a time $T>0$ such that

{\bf (1)\; Existence and Uniqueness:} the transport type equation \eqref{TR} has a unique solution $u(t,x)=\mathbf{S}_{t}(u_0)\in\mathcal{C}_w\left(0, T ; B_{p, \infty}^s(\R)\right) \cap C^{0,1}\left([0, T] ; B_{p, \infty}^{s-1}(\R)\right)$;

{\bf (2)\; Non-Continuous Dependence:} the data-to-solution map of \eqref{TR} obtained  in {\bf (1)}
is not continuous from any bounded subset in $B^s_{p,\infty}$ into $L^\infty_T(B^s_{p,\infty}(\R))$. More precisely, there exists two sequences of solutions $\mathbf{S}_t(u^n_0)$ and $\mathbf{S}_t(u_0)$ such that
\bbal
\lim_{n\rightarrow \infty}\|u^n_0-u_0\|_{B^s_{p,\infty}(\R)}= 0,
\end{align*}
but
\bbal
\liminf_{t_n\to 0}\|\mathbf{S}_{t_n}(u^n_0)-\mathbf{S}_{t_n}(u_0)\|_{B^s_{p,\infty}(\R)}\geq \eta>0.
\end{align*}

\end{theorem}
\subsection{Construction of Initial Data}
Before constructing the initial data for the 1-D transport type equation on the real line, we need to introduce smooth, radial cut-off functions to localize the frequency region. We define an even, real-valued and non-negative function $\widehat{\phi}\in \mathcal{C}^\infty_0(\mathbb{R})$ with values in $[0,1]$ which satisfies
\bal\label{fi}
\widehat{\phi}(\xi)=
\bca
1, \quad \mathrm{if} \ |\xi|\leq \frac{1}{4},\\
0, \quad \mathrm{if} \ |\xi|\geq \frac{1}{2}.
\eca
\end{align}
It is easy to check that $\phi(x)=(2\pi)^{-1}\mathcal{F}^{-1}(\widehat{\phi}(\xi))$ is a real-valued and continuous function on $\R$. Also, it holds $\|\phi\|_{L^\infty}=\phi(0)>0$ and $\|\phi\|_{L^p}\approx 1$ with $1\leq p<\infty$. Then there exists some integer number $N_0$ which can be chosen large enough such that
$\phi(x)\geq \phi(0)/{2}$ for any $x\in [0,2\pi2^{-N_0}].$ With this, we can define the real-valued initial data $u_0$.
\begin{definition}\label{def} Let $s> 1$. We define initial data $u_0$  as follows
\bbal
u_0(x):=
\sum^{\infty}_{j=3}2^{-js}\tilde{\phi}(x)\cos\f(\frac{11}{8}2^jx\g), \quad  x\in \mathbb{R},
\end{align*}
where we denote  $\mathrm{\gamma}(s):=2^{2s}(2^s-1)>1$  and
\bbal
\tilde{\phi}(x):=
\mathrm{\gamma}(s)\frac{\phi(x)}{\phi(0)}.
\end{align*}
\end{definition}

\subsection{Technical Lemmas}
In this subsection we establish some technical lemmas which will be used in the sequel.
\begin{lemma}\label{js1} Let $u_0$ be given by Definition \ref{def}. Assume that $(s,p)$ satisfies $s>1+\frac1p$ with $1\leq p\leq \infty$.
Then there exists some sufficiently large integer $N_0$ and constant $C_s>0$ independent of $N_0$  such that for arbitrarily  large integer $n\gg N_0$
\begin{align}
&\|u_0\|_{L^\infty(\mathbb{R})}\leq u_0(0)=1,\quad\|u'_0\|_{L^\infty(\mathbb{R})}\leq C_s,\label{zy1}\\
&\Delta_nu_0=2^{-ns}\tilde{\phi}(x)\cos\f(\frac{11}{8}2^nx\g),\quad\Delta_n\tilde{\phi}=0, \label{zy2}\\
&u_0\in W^{1,\infty}(\mathbb{R})\cap B^s_{p,\infty}(\mathbb{R}) \quad \text{with}\quad \|u_0\|_{W^{1,\infty}(\mathbb{R})\cap B^s_{p,\infty}(\mathbb{R})}\leq C_s.\label{zy3}
\end{align}
\end{lemma}
\begin{proof}
\eqref{zy1} is obvious. Set $f_j(x):=\tilde{\phi}(x)\cos\f(\frac{11}{8}2^jx\g)$.
Straightforward computations yield
\bbal
\widehat{f_j}(\xi)=\frac{\mathrm{\gamma}(s)}{\phi(0)}
\left[\widehat{\phi}\left(\xi-\frac{11}{8}2^j\right)+\widehat{\phi}\left(\xi+\frac{11}{8}2^j\right)\right],
\end{align*}
which implies
\bal\label{s}
\mathrm{supp} \ \widehat{f_j}&\subset \left\{\xi\in\R: \ \frac{11}{8}2^j-\fr12\leq |\xi|\leq \frac{11}{8}2^j+\fr12\right\}
\subset \left\{\xi\in\R: \ \frac{4}{3}2^j\leq |\xi|\leq \frac{3}{2}2^j\right\}.
\end{align}
Recalling that $\mathrm{supp}\ \widehat{\phi}\subset [0,\fr12]$ and $\varphi\equiv 1$ for $\frac43\leq |\xi|\leq \frac32$, we obtain that  $\Delta_n\phi=0$ for large $n$ and for any  $k\in\{-1,0\}\cup \mathbb{N}^+$
\begin{align*}{\Delta_kf_j=}
\begin{cases}
f_k, &\mathrm{if} \; j=k,\\
0, &\mathrm{if} \; j\neq k.\end{cases}
\end{align*}
Thus \eqref{zy2} holds.
Then from \eqref{s}, we deduce that
\bbal
\|u_0\|_{B^s_{p,\infty}(\mathbb{R})}&=\sup_{k\geq-1}2^{ks}\|\Delta_ku_0\|_{L^{p}(\mathbb{R})}
= \sup_{k\geq0}\f\|\tilde{\phi}(x)\cos\f(\frac{11}{8}2^kx\g)\g\|_{L^p(\mathbb{R})}
\leq C_{s}\|\tilde{\phi}\|_{L^p(\mathbb{R})},
\end{align*}
which completes the proof of Lemma \ref{js1}.
\end{proof}

\begin{lemma}\label{tl2}
Let $p\in[1,\infty)$ and $g(x)\in W^{1,\infty}(\R)$. Denote $\lambda_n:= \frac{11}{8}2^n$ and $t_n:=\frac{8}{11}\pi n2^{-n}$. Then, we have for $1\ll n\in \mathbb{N}^+$ and $t\in[0,1]$
\bbal
&\f\|\phi(x-tg(x))\g\|_{L^p(\R)}\leq C,\\
&\f\|\cos\f(\lambda_n(x-t_ng(x))\g)\g\|_{L^p([0,2\pi])}\geq \frac1{4}.
\end{align*}
\end{lemma}
\begin{proof} Since $\phi$ is a Schwartz function, we have for $100\leq M\in \mathbb{N}^+$
\bbal
|\phi(x)|\leq C(1+|x|)^{-M},
\end{align*}
which means that
\bbal
\int_{\R}\f|\phi(x-tg(x))\g|^p\dd x&\leq C\int_{\R}(1+|x-tg(x)|)^{-Mp}\dd x
\\&\leq C2^{Mp}\int_{|x|\geq 2\|g\|_{L^\infty}+2}|x|^{-Mp}\dd x+C\int_{|x|\leq 2\|g\|_{L^\infty}+2}1\dd x\leq C,
\end{align*}
where we have used the simple fact that $|x-tg(x)|\geq |x|-t|g(x)|\geq |x|-\|g\|_{L^\infty}$.

By some simple variable changes, we obtain
\bbal
\int^{2\pi}_0\f|\cos\f(\lambda_n(x-t_ng(x))\g)\g|^p\dd x
&=\frac{1}{\lambda_n}\int^{2\pi \lambda_n}_0\f|\cos\f(x-\lambda_nt_n g\f(\frac{x}{\lambda_n}\g)\g)\g|^p\dd x
\\&=\frac{1}{\lambda_n}\sum^{\lambda_n}_{j=1}\int^{2\pi j }_{2\pi(j-1)}\f|\cos\f(x-\lambda_nt_n g\f(\frac{x}{\lambda_n}\g)\g)\g|^p\dd x
\\&=\frac{1}{\lambda_n}\sum^{\lambda_n}_{j=1}\int^{2\pi }_{0}\f|\cos\f(x-\lambda_nt_n g\f(\frac{x}{\lambda_n}+\frac{2\pi (j-1)}{\lambda_n}\g)\g)\g|^p\dd x\\&\geq \frac{1}{\lambda_n}\sum^{\lambda_n}_{j=1} (I-J)^p,
\end{align*}
where
\bbal
&I:=\f\|\cos\f(x-\lambda_nt_n g\f(\frac{2\pi(j-1)}{\lambda_n}\g)\g)\g\|_{L^p([0,2\pi])}=\f\|\cos x\g\|_{L^p([0,2\pi])},\\
&J:=\f\|\cos\f(x-\lambda_nt_n g\f(\frac{x}{\lambda_n}+\frac{2\pi(j-1)}{\lambda_n}\g)\g)-\cos\f(x-\lambda_nt_n g\f(\frac{2\pi(j-1)}{\lambda_n}\g)\g)\g\|_{L^p([0,2\pi])}.
\end{align*}
Notice that for $x\in[0,2\pi]$
\bbal
&\f|\cos\f(x-\lambda_nt_n g\f(\frac{x}{\lambda_n}+\frac{2\pi(j-1)}{\lambda_n}\g)\g)-\cos\f(x-\lambda_nt_n g\f(\frac{2\pi(j-1)}{\lambda_n}\g)\g)\g|\\
&\leq \lambda_nt_n\f|g\f(\frac{x}{\lambda_n}+\frac{2\pi(j-1)}{\lambda_n}\g)-g\f(\frac{2\pi(j-1)}{\lambda_n}\g)\g|
\\&\leq t_nx\f|g'\f(\xi_{j,n,x}\g)\g|,\qquad \xi_{j,n,x}\in[0,2\pi]
\\&\leq Ct_n,
\end{align*}
which gives that
$J\leq  Ct_n$, and thus for large $n$ we have
\bbal
I-J\geq \frac{1}{2}\f\|\cos x\g\|_{L^p([0,2\pi])}\geq \frac{1}{4}.
\end{align*}
This completes the proof of Lemma \ref{tl2}.
\end{proof}
As an application, we have the higher dimensional version which will be used to the Euler equations.
\begin{lemma}\label{tl3}
Let $p\in[1,\infty)$ and $g(x)\in W^{1,\infty}(\R^d)$ with $d\geq2$. Then we have for some $N_0\in \mathbb{N}^+$ and $n\gg1$
\bal\label{hh}
\f\|\cos\f(\lambda_n(x_1-t_ng(x))\g)\g\|_{L^p([0,2\pi2^{-N_0}]^d)}\geq 2^{-\frac{d}{p}N_0-2}.
\end{align}
\end{lemma}
\begin{proof}
Setting $(\tilde{\lambda}_n,\tilde{t}_n)=(2^{-N_0}\lambda_n,2^{N_0}t_n)$ and $\tilde{g}(x_1)=g(2^{-N_0}x_1,\tilde{x}))$ with $\tilde{x}=(x_2,\ldots,x_d)$. As in Lemma \ref{tl2}, one has
\bbal
\f\|\cos\f(\lambda_n(x_1-t_ng(x))\g)\g\|^p_{L_{x_1}^p([0,2\pi2^{-N_0}])}&= 2^{-N_0} \int^{2\pi}_{0}\f|\cos\f(\tilde{\lambda}_n(x_1-\tilde{t}_n\tilde{g}(x_1)\g)\g|^p\dd x_1
\geq \frac{1}{4^p}2^{-N_0},
\end{align*}
and  thus
\bbal
\text{Left of}\; \eqref{hh}&=\f\|\f\|\cos\f(\lambda_n(x_1-t_ng(x))\g)\g\|_{L^p_{x_1}([0,2\pi2^{-N_0}])}\g\|_{L_{\tilde{x}}^p([0,2\pi2^{-N_0}]^{d-1})}\geq \frac{1}{4}2^{-\frac{d}{p}N_0}.
\end{align*}
This completes the proof of Lemma \ref{tl3}.
\end{proof}

\subsection{Constructions of Approximation Solutions}

We begin with a series of approximation lemmas.
\begin{proposition}\label{3pr1} Assume that $u_0\in B^{s}_{p,\infty}(\mathbb{R})$ with $s>1+1/p,1\leq p\leq \infty$.
Suppose that $u$ is the solution of the original equation \eqref{TR} and $u^{\rm{ap},1}$ solves the following linear system
\bal\label{3h1}
\bca
\pa_t u^{\rm{ap},1}+u_0\pa_xu^{\rm{ap},1}=F(u_0), \\
u^{\rm{ap},1}(0,x)=u_0(x).
\eca
\end{align}
For small $t>0$, we have for some positive constant $C$ independent of $t$
\begin{align*}
\|u-u^{\rm{ap},1}\|_{B_{p,\infty}^{s-1}(\mathbb{R})}\leq Ct^2.
\end{align*}
\end{proposition}
\begin{proof} By the Newton-Leibniz formula and Lemma \ref{lp}, one has
\bbal
\|u(t)-u_0\|_{B^{s-1}_{p,\infty}(\mathbb{R})}\leq \int_0^t\|u\pa_xu-F(u)\|_{B^{s-1}_{p,\infty}(\mathbb{R})}\dd\tau\leq Ct.
\end{align*}
We denote $w:=u-u^{\rm{ap},1}$. From \eqref{3h1}, one has
\bal\label{3h2}
\bca
\pa_tw+u\pa_xw=-(u-u_0)\pa_xu^{\rm{ap},1}+F(u)-F(u_0),\\
w(t=0)=0.
\eca
\end{align}
Notice that
\begin{align*}
\|F(u)-F(u_0)\|_{B_{p, \infty}^{s-1}} & \leq C\|u-u_0\|_{B_{p, \infty}^{s-1}},
\end{align*}
and applying Lemma \ref{zzgj} to \eqref{3h2}, we have
\bbal
\|w\|_{B_{p,\infty}^{s-1}(\mathbb{R})}\leq&~\int_0^t\|(u-u_0)\pa_xu^{\rm{ap},1}\|_{B_{p,\infty}^{s-1}(\mathbb{R})}+\|F(u)-F(u_0)\|_{B_{p,\infty}^{s-1}(\mathbb{R})}\dd\tau\\
\leq&~ C\int_0^t\|u-u_0\|_{B_{p,\infty}^{s-1}(\mathbb{R})}\dd\tau\leq Ct^2.
\end{align*}
Then we complete the proof of Proposition \ref{3pr1}.
\end{proof}
\begin{proposition}\label{3pr2} Assume that $u_0\in B^{s}_{p,\infty}(\mathbb{R})$ with $s>1+1/p,1\leq p\leq \infty$.
Suppose that $u^{\rm{ap},1}$ is the solution of \eqref{3h1} and $u^{\rm{ap},2}$ solves the following linear system
\bal\label{h1}
\bca
\pa_t u^{\rm{ap},2}+u_0\pa_xu^{\rm{ap},2}=0, \\
u^{\rm{ap},2}(0,x)=u_0(x).
\eca
\end{align}
For small $t>0$, we have for some positive constants $C$ independent of $t$
\begin{align*}
\f\|u^{\rm{ap},1}-u^{\rm{ap},2}\g\|_{B^{s}_{p,\infty}(\mathbb{R})}\leq Ct.
\end{align*}
\end{proposition}
\begin{proof}  We denote $w:=u^{\rm{ap},1}-u^{\rm{ap},2}$. From \eqref{3h1} and  \eqref{h1}, one has
\bal\label{h2}
\bca
\pa_t w+u_0\pa_x w=F(u_0), \\
w(0,x)=0.
\eca
\end{align}
Applying Lemma \ref{zzgj} to \eqref{h2}, we have
\bbal
\|w\|_{B_{p,\infty}^{s}(\mathbb{R})}\leq&~ \int_0^t\|F(u_0)\|_{B_{p,\infty}^{s}(\mathbb{R})}\dd\tau
\leq Ct,
\end{align*}
which completes the proof of Proposition \ref{3pr2}.
\end{proof}
\begin{proposition}\label{3pr3} Assume that $u_0\in B^{s}_{p,\infty}(\mathbb{R})$ with $s>1+1/p,1\leq p\leq \infty$.
Suppose that $u^{\rm{ap},2}$ is the solution of \eqref{h1} and $u^{\rm{ap},3}_n$ satisfies the following linear system
\bal\label{h3}
\bca
\pa_t u^{\rm{ap},3}_n+u_0\pa_xu^{\rm{ap},3}_n=0, \\
u^{\rm{ap},3}_n(0,x)=\Delta_nu_0(x).
\eca
\end{align}
For small $t>0$, we have for some positive constant $C$ independent of $t$ and $n$
\begin{align*}
&2^{ns}\f\|\Delta_nu^{\rm{ap},2}-u^{\rm{ap},3}_n\g\|_{L^p(\mathbb{R})}\leq Ct.
\end{align*}
\end{proposition}
\begin{proof}  We denote $w:=\Delta_nu^{\rm{ap},2}-u^{\rm{ap},3}_n$. From \eqref{h1} and \eqref{h3}, one has
\bal\label{h4}
\bca
\pa_tw+u_0\pa_xw=-[\Delta_n,u_0]\pa_xu^{\rm{ap},2},\\
w(0,x)=0.
\eca
\end{align}
Taking the inner product of Eq. \eqref{h4} with $|w|^{p-2}w$ and using Lemma \ref{jhz}, we obtain
\bbal
2^{ns}\|w\|_{L^p(\mathbb{R})}\leq&~ \int_0^t2^{ns}\|\pa_xu_0\|_{L^{\infty}(\mathbb{R})}\|w\|_{L^p(\mathbb{R})}\dd\tau+\int_0^t2^{ns}\f\|[\Delta_n,u_0]\pa_xu^{\rm{ap},2}\g\|_{L^{p}(\mathbb{R})}\dd\tau\\
\leq&~C\int_0^t2^{ns}\|w\|_{L^p(\mathbb{R})}\dd\tau+Ct.
\end{align*}
By Gronwall's inequality, we complete the proof of Proposition \ref{3pr3}.
\end{proof}
\begin{proposition}\label{3pr4} Assume that $u_0\in B^{s}_{p,\infty}(\mathbb{R})$ with $s>1+1/p$ and $1\leq p\leq \infty$. Denoting $u^{\rm{ap},4}_n=(\Delta_nu_0)(x-tu_0(x))$. Then $u^{\rm{ap},4}_n$ satisfies the following system
\bal\label{h5}
\bca
\pa_tu^{\rm{ap},4}_n+u_0\pa_xu^{\rm{ap},4}_n=-tu_0\pa_xu_0(\pa_x\Delta_nu_0)(x-tu_0(x)), \\
u^{\rm{ap},4}(0,x)=\Delta_nu_0(x).
\eca
\end{align}
For small $t>0$, we have for some positive constant $C$ independent of $t$ and $n$
\begin{align*}
&\f\|u^{\rm{ap},4}_n-u^{\rm{ap},3}_n\g\|_{L^p(\mathbb{R})}\leq C\int_0^t\tau\|(\pa_x\Delta_nu_0)(x-\tau u_0(x))\|_{L^{p}(\mathbb{R})}\dd\tau.
\end{align*}
In particular, assume that $u_0$ is given by Definition \ref{def}, then
\begin{align*}
&2^{ns}\f\|u^{\rm{ap},4}_n-u^{\rm{ap},3}_n\g\|_{L^p(\mathbb{R})}\leq Ct^22^n.
\end{align*}
\end{proposition}
\begin{proof}  We denote $w:=u^{\rm{ap},4}_n-\Delta_nu^{\rm{ap},3}$. From \eqref{h3} and \eqref{h5}, one has
\bal\label{h6}
\bca
\pa_tw+u_0\pa_x w=-tu_0\pa_xu_0(\pa_x\Delta_nu_0)(x-tu_0(x)),\\
w(0,x)=0.
\eca
\end{align}
Taking the inner product of Eq. \eqref{h6} with $|w|^{p-2}w$, we obtain
\bbal
\|w\|_{L^p(\mathbb{R})}\leq&~ \int_0^t\|\pa_xu_0\|_{L^{\infty}(\mathbb{R})}\|w\|_{L^p(\mathbb{R})}\dd\tau+\int_0^t\tau\|u_0\pa_xu_0(\pa_x\Delta_nu_0)(x-\tau u_0(x))\|_{L^{p}(\mathbb{R})}\dd\tau\\
\leq&~C\int_0^t\|w\|_{L^p(\mathbb{R})}\dd\tau+C\int_0^t\tau\|(\pa_x\Delta_nu_0)(x-\tau u_0(x))\|_{L^{p}(\mathbb{R})}\dd\tau.
\end{align*}
By Gronwall's inequality, we complete the proof of Proposition \ref{3pr4}.
\end{proof}

\subsection{Estimations of Approximate Solutions}

We define an even, real-valued and non-negative function $\widehat{\psi}\in \mathcal{C}^\infty_0(\mathbb{R})$ with values in $[0,1]$ which satisfies
\bbal
\widehat{\psi}(\xi)=
\bca
1, \quad \mathrm{if} \ \frac12\leq |\xi|\leq \frac{5}{8},\\
0, \quad \mathrm{if} \ |\xi|\geq \frac{3}{4} \ \mathrm{or} \ |\xi|\leq  \frac38.
\eca
\end{align*}
It is easy to check that $\psi(x)=(2\pi)^{-1}\mathcal{F}^{-1}(\widehat{\psi}(\xi))$ is a real-valued and continuous function on $\R$ and $\|\psi\|_{L^\infty}=\psi(0)>0$.
\begin{definition}\label{DEF} Let $u_0$ be given by Definition \ref{def} and define a perturbation sequence of initial data $u_0$
\bbal
u^n_0(x):=
u_0(x)+\frac1n\Phi(x),\quad \text{where} \quad \Phi(x):=
\frac{\psi(x)}{\psi(0)},\quad x\in \mathbb{R}.
\end{align*}
\end{definition}
Next, we need compute the difference between the approximate solution $u^{\rm{ap},4}_{n,1}$ of original solution and approximate solution $u^{\rm{ap},4}_{n,2}$ of perturbed solution.
\begin{proposition}\label{3pr5} Setting
\bbal
u^{\rm{ap},4}_{n,1}(t,x)=(\Delta_nu_0)(x-tu_0(x)) \quad\text{and}\quad u^{\rm{ap},4}_{n,2}(t,x)=(\Delta_nu^n_0)(x-tu^{n}_0(x)).
\end{align*}
For small $t>0$ and large $N_0$, we have for $n\gg N_0$
\begin{align*}
2^{ns}\f\|\f(u^{\rm{ap},4}_{n,1}-u_{n,2}^{\rm{ap},4}\g)(t,x)\g\|_{L^p(\mathbb{R})}\geq 2^{-\frac{N_0}p-3}.
\end{align*}
\end{proposition}
\begin{proof}
Due to $\mathrm{supp}\ \widehat{\Phi}\subset [\fr38,\fr34]$, which implies that $\Delta_n\Phi=0$ for $n\gg N_0$, from Lemma \ref{js1}, one has
\bbal
\Delta_n
u_0=\Delta_n
u^n_0=2^{-ns}\tilde{\phi}(x)\cos\f(\frac{11}{8}2^nx\g), \qquad n\gg N_0.
\end{align*}
Thus we have
\bbal
2^{ns}u^{\rm{ap},4}_{n,2}(t,x)&=2^{ns}(\Delta_n u^n_0)\f(x-tu_0^n(x)\g)\\
&=\tilde{\phi}(x)\cos\f(\frac{11}{8}2^n(x-tu_0(x))-\frac{11}{8}2^n\frac{t}n\Phi(x)\g)+K_1(t),
\\ 2^{ns}u^{\rm{ap},4}_{n,1}(t,x)&=2^{ns}(\Delta_n u_0)(x-tu_0(x))\\
&=\tilde{\phi}(x)\cos\f(\frac{11}{8}2^n(x-tu_0(x))\g)+K_2(t),
\end{align*}
where
\bbal
&K_1(t):=\f[\tilde{\phi}\f(x-tu_0^n(x)\g)-\tilde{\phi}(x)\g]\cos\f(\frac{11}{8}2^n(x-tu^n_0(x))\g),\\
&K_2(t):=\f[\tilde{\phi}\f(x-tu_0(x)\g)-\tilde{\phi}(x)\g]\cos\f(\frac{11}{8}2^n(x-tu_0(x))\g).
\end{align*}
Letting $\frac{11}{8}2^{n}t_n=n\pi$, then one has
\bbal
&\quad 2^{ns}\f(u^{\rm{ap},4}_{n,2}(t_n)-u_{n,1}^{\rm{ap},4}(t_n)\g)
\\&=\tilde{\phi}(x)\f[\cos\f(\frac{11}{8}2^n(x-t_nu_0(x))-\pi\Phi(x)\g)-\cos\f(\frac{11}{8}2^n(x-t_nu_0(x))\g)\g]
+K_1(t_n)-K_2(t_n),
\end{align*}
which implies
\bal\label{H}
&\quad 2^{ns}\f\|u^{\rm{ap},4}_{n,2}(t_n)-u_{n,1}^{\rm{ap},4}(t_n)\g\|_{L^p(\R)}
\nonumber\\&\geq \f\|\tilde{\phi}(x)\f[\cos\f(\frac{11}{8}2^n(x-t_nu_0(x))-\pi\Phi(x)\g)
-\cos\f(\frac{11}{8}2^n(x-t_nu_0(x))\g)\g]\g\|_{L^p(\R)}\nonumber\\&\quad
-\f\|K_1\g\|_{L^p(\R)}-\f\|K_2\g\|_{L^p(\R)}.
\end{align}
Obviously,
\bbal
\f\|K_1(t_n)\g\|_{L^p(\R)}&\leq \|\tilde{\phi}\f(x-t_nu_0^n(x)\g)-\tilde{\phi}(x)\|_{L^p(\R)}
\leq C\|t_nu_0^n(x)\|_{L^p(\R)}\leq Ct_n,
\end{align*}
and similarly
\bbal
\f\|K_2(t_n)\g\|_{L^p(\R)}&\leq Ct_n.
\end{align*}

{\bf Case for $p=\infty$.}\; Taking $x=0$ and using the fact $u_0(0)=\Phi(0)=1$, one has for $n\gg 1$
\bbal
2^{ns}\f\|u^{\rm{ap},4}_{n,2}(t_n)-u_{n,1}^{\rm{ap},4}(t_n)\g\|_{L^\infty(\R)}&\geq  2\tilde{\phi}(0)\f|\cos\f(\frac{11}{8}2^n(-t_nu_0(0))\g)\g|-Ct_n
\\&\geq  2\mathbf{\gamma}(s)\f|\cos\f(\frac{11}{8}2^n(-t_nu_0(0))\g)\g|-Ct_n
\\&\geq \mathbf{\gamma}(s)\geq 1.
\end{align*}

{\bf Case for $1\leq p<\infty$.}\;
Noticing that there exists $N_0>0$ such that
\bbal
\forall x\in[0,2\pi2^{-N_0}], \qquad \varphi(x)\geq \frac12\varphi(0) \quad \Longleftrightarrow\quad \tilde{\phi}(x)\geq\fr12,
\end{align*}
and due to $\Phi(0)=1$, we have
\bbal
&\quad\f\|\tilde{\phi}(x)\f[\cos\f(\frac{11}{8}2^n(x-t_nu_0(x))-\pi\Phi(x)\g)
-\cos\f(\frac{11}{8}2^n(x-t_nu_0(x))\g)\g]\g\|_{L^p([0,2\pi2^{-N_0}])}\nonumber\\
&\geq \frac1{2}\f\|\cos\f(\frac{11}{8}2^n(x-t_nu_0(x))-\pi\Phi(x)\g)
-\cos\f(\frac{11}{8}2^n(x-t_nu_0(x))\g)\g\|_{L^p([0,2\pi2^{-N_0}])}\nonumber\\
&\geq I_1-\frac12I_2,
\end{align*}
where
\bbal
&I_I:=\f\|\cos\f(\frac{11}{8}2^n(x-t_nu_0(x))\g)\g\|_{L^p([0,2\pi2^{-N_0}])},\\
&I_2:=\f\|\cos\f(\frac{11}{8}2^n(x-t_nu_0(x))-\pi\Phi(x)\g)
-\cos\f(\frac{11}{8}2^n(x-t_nu_0(x))-\pi\Phi(0)\g)\g\|_{L^p([0,2\pi2^{-N_0}])}.
\end{align*}
Using the basic fact that $|\cos a-\cos b|\leq |a-b|$, then we have
\bbal
&\quad \f|\cos\f(\frac{11}{8}2^n(x-t_nu_0(x))-\pi\Phi(x)\g)
-\cos\f(\frac{11}{8}2^n(x-t_nu_0(x))
-\pi\Phi(0)\g)\g|
\\&\leq\pi\f|\Phi(x)-\Phi(0)\g|=\pi x\f\|\Phi'\g\|_{L^\infty(\R)}\\&
\leq \pi x\f\|\xi\widehat{\Phi}(\xi)\g\|_{L^1(\R)}\leq C2^{-N_0},\quad \forall x\in [0,2\pi2^{-N_0}]
\end{align*}
which leads to
\bbal
I_2
\leq C2^{-(1+\frac{1}{p})N_0}.
\end{align*}
By Lemma \ref{tl2}, we also have
\bbal
I_1\geq 2^{-\frac{N_0}p-2}.
\end{align*}
Inserting the above into \eqref{H} and choosing large enough $N_0$, we complete the proof of Proposition \ref{3pr5}.
\end{proof}

\subsection{Non-Continuous Dependence}

Now, we can prove Theorem \ref{th3}.

{\bf Difference of initial data.}\; Obviously, we see that
\bbal
\|u^{n}_{0}-u_0\|_{B^s_{p,\infty}(\R)}&= \frac1{n}2^{-s}\f\|\frac{\psi(x)}{\psi(0)}\g\|_{L^p(\R)}\to0,\quad \text{as}\;n\to\infty.
\end{align*}

{\bf Difference of solutions.}\; For simplicity, we set $u_n(t_n)=\mathbf{S}_{t_n}(u^n_0)$  and $u(t_n)=\mathbf{S}_{t_n}(u_0)$. Noticing that
\bbal
\Delta_nu&=\Delta_n(u-u^{\rm{ap},1})+\Delta_n(u^{\rm{ap},1}-u^{\rm{ap},2})+\Delta_nu^{\rm{ap},2}-u^{\rm{ap},3}+u^{\rm{ap},3}_n-u^{\rm{ap},4}_n+u^{\rm{ap},4}_n,
\end{align*}
and using Propositions \ref{3pr1}-\ref{3pr5}, we deduce that for large $N_0$ and $n\gg N_0$
\bbal
\|u_n(t_n)-u(t_n)\|_{B^s_{p,\infty}(\R)}&\geq2^{ns}\|\Delta_n(u_n-u)(t_n)\|_{L^p(\R)}-Ct^2_n2^n-Ct_n-\frac{C}{n}\\
&\geq 2^{-\frac{N_0}{p}-3}-Ct^2_n2^n-Ct_n-\frac{C}{n}.
\end{align*}
Letting $n\to\infty$, we complete the proof of Theorem \ref{th3}.
\begin{remark}
Theorem \ref{th3} also holds for the periodic 1-D transport type equation if we choose the periodic initial data
$u_0(x):=
2^{2s}(2^s-1)\sum^{\infty}_{j=3}2^{-js}\cos\f(\frac{11}{8}2^jx\g)$ and $u^n_0(x)=u_0(x)+\frac1n$.
\end{remark}
\section{Proof of Theorem \ref{th1}: Periodic case}\label{sec4}

For the period case, an fundamental observe is that the vector field
\bbal
&u(t,x)=\vec{e}_1+\vec{e}_2\sum\limits^\infty_{j=3}2^{-js}\cos \f(\frac{11}{8}2^{j}(x_1-t)\g),
\end{align*}
is explicit periodic solutions of the incompressible Euler equations
\bbal\pa_t u+u\cdot \nabla u+\nabla P=0, \quad
\mathrm{div\,} u=0,
\end{align*}
with $P = 0$, i.e. this is a pressureless flow.

Let $\lambda_n=1+\fr{1}{n}$, we consider
\bbal
&u(t,x)=\vec{e}_1+\vec{e}_2f(x_1-t)=\f(1,f(x_1-t),0,\ldots,0\g), \\
&u_n(t,x)=\lambda_n\vec{e}_1+\vec{e}_2f\f(x_1-\lambda_nt\g)=\f(\lambda_n,f\f(x_1-\lambda_nt\g),0,\ldots,0\g),
\end{align*}
where $f(x_1)$ is a bounded real-valued periodic function of one variable with the following form
\bbal
f(x_1)=\sum\limits^\infty_{j=3}2^{-js}\cos \f(\frac{11}{8}2^{j}x_1\g),\quad x_1\in \T.
\end{align*}
It is not difficult to verify that both $u(t,x)$ and $u_n(t,x)$ are periodic solutions to the Euler equations \eqref{E} with initial data respectively
\bbal
&u_0(x)=\vec{e}_1+\vec{e}_2f(x_1)=(1,f(x_1),0,\cdots,0),\\
&u_{0,n}(x)=\lambda_n\vec{e}_1+\vec{e}_2f(x_1)=(\lambda_n,f(x_1),0,\cdots,0).
\end{align*}

{\bf Difference of initial data.}\; Obviously, we see that
\bbal
\|u_{0,n}-u_0\|_{B^s_{p,\infty}(\T^d)}= \frac1n\|\vec{e}_1\|_{B^s_{p,\infty}(\T^d)}\leq \frac{C}n\to0,\quad \text{as}\;n\to\infty,
\end{align*}
where we have used $\Delta_j1=0$ if $j\neq -1$.

{\bf Difference of solutions.}\; Obviously,
$$u_n(t,x)-u(t,x)=\f(\fr{1}{n},f(x_1-\lambda_nt)-f(x_1-t),0,\ldots,0\g).$$
Next we shall prove that \bbal
\|u_n(t,x)-u(t,x)\|_{B^s_{p,\infty}(\T^d)}\nrightarrow0,\quad \text{as}\;n\to\infty.
\end{align*}
A straightforward computation shows that
\bbal
f(x_1-\lambda_nt)-f(x_1-t)
&=\sum\limits^\infty_{j=3}2^{-js}\f[\cos \f(\frac{11}{8}2^{j}x_1\g)\mathbf{c}_j(t)+\sin \f(\frac{11}{8}2^{j}x_1\g)\mathbf{d}_j(t)\g],
\end{align*}
where
\bbal
&\mathbf{c}_j(t)=\cos \f(\frac{11}{8}2^{j}\lambda_nt\g)-\cos \f(\frac{11}{8}2^{j}t\g),
\\ &\mathbf{d}_j(t)=\sin \f(\frac{11}{8}2^{j}\lambda_nt\g)-\sin \f(\frac{11}{8}2^{j}t\g).
\end{align*}
Then, we have for some $n$ large enough
\bbal
2^{ns}\De_n[f(x_1-\lambda_nt)-f(x_1-t)]&=\cos \f(\frac{11}{8}2^{n}x_1\g)\mathbf{c}_n(t)
 +\sin \f(\frac{11}{8}2^{n}x_1\g)\mathbf{d}_n(t).
\end{align*}
Letting $\frac{11}{8}2^{n}t_n=n\pi$, then one has $\mathbf{c}_n(t_n)=\pm2$ and $\mathbf{d}_n(t_n)=0$, which give that
\bbal
&\quad 2^{ns}\f\|\De_n[f(x_1-\lambda_nt_n)-f(x_1-t_n)]\g\|_{L^p(\T^d)}
=2(2\pi)^{d-1}\f\|\cos \f(\frac{11}{8}2^{n}x_1\g)\g\|_{L^p(\mathbb{T})}.
\end{align*}
Thus, we have for some $n$ large enough
\bbal
\|u_n(t_n,x)-u(t_n,x)\|_{B^s_{p,\infty}(\T^d)}&\geq \|f(x_1-\lambda_nt_n)-f(x_1-t_n)\|_{B^s_{p,\infty}(\T^d)}-\fr1{n}\\
&\geq2^{ns}\f\|\De_n[f(x_1-\lambda_nt)-f(x_1-t_n)]\g\|_{L^p(\T^d)}-\fr1{n}\\&\geq 2c_0(2\pi)^{d-1}-\fr1{n}.
\end{align*}
Letting $n\to\infty$, we complete the proof of Theorem \ref{th1}.
\section{Proof of Theorem \ref{th1}: Non-periodic case}\label{sec5}
\subsection{Constructions of Approximation Solutions}

Due to the incompressibility condition, the pressure can be eliminated from \eqref{E}. In fact, applying the Leray operator $\mathcal{P}$ to \eqref{E}, then we have
\begin{align}\label{YH1}
\begin{cases}
\pa_t u+\mathcal{P}(u\cdot \nabla u)=0, \\
u(0,x)=u_0(x),
\end{cases}
\end{align}
or equivalently,
\begin{align}\label{YH2}
\begin{cases}
\pa_t u+u\cdot \nabla u=\mathcal{Q}(u\cdot \nabla u), \\
u(0,x)=u_0(x).
\end{cases}
\end{align}
Following the idea as proceed above, we need to establish a series of estimations of actual and approximate solutions for the Euler equations.
\begin{proposition}\label{4pr1} Assume that $(s,p)$ satisfies \eqref{con1} and $u_0\in B^{s}_{p,\infty}(\R^d)$ with $\div\ u_0=0$.
Suppose that $u$ is the solution of the original system \eqref{YH2} and $u^{\rm{ap},i}(i=1,2,3)$ solves the following systems
\bal\label{e1}
&\bca
\pa_t u^{\rm{ap},1}+u_0\cdot \nabla u^{\rm{ap},1}=\mathcal{Q}(u_0\cdot \nabla u_0), \\
u^{\rm{ap},1}(0,x)=u_0(x),
\eca\\
&
\bca
\pa_t u^{\rm{ap},2}+u_0\cdot \nabla u^{\rm{ap},2}=0, \\
u^{\rm{ap},2}(0,x)=u_0(x),
\eca\\
&
\bca
\pa_t u^{\rm{ap},3}_n+u_0\cdot \nabla u^{\rm{ap},3}_n=0, \\
u^{\rm{ap},3}_n(0,x)=\Delta_nu_0(x).
\eca
\end{align}
Denoting $u^{\rm{ap},4}_n=(\Delta_nu_0)(x-tu_0(x))$ which satisfies the following system
\bal\label{euler4}
\bca
\pa_tu^{\rm{ap},4}_n+u_0\cdot \nabla u^{\rm{ap},4}_n=-tu_0\cdot\nabla u_0\cdot(\nabla\Delta_nu_0)(x-tu_0(x)), \\
u^{\rm{ap},4}(0,x)=\Delta_nu_0(x).
\eca
\end{align}
For small $t>0$, we have for some positive constant $C$ independent of $t$ and $n$
\begin{align}
&\f\|u-u^{\rm{ap},1}\g\|_{B_{p,\infty}^{s-1}(\R^d)}\leq Ct^2,\\
&\f\|u^{\rm{ap},1}-u^{\rm{ap},2}\g\|_{B^{s}_{p,\infty}(\R^d)}\leq Ct,\\
&2^{ns}\f\|\Delta_nu^{\rm{ap},2}-u^{\rm{ap},3}_n\g\|_{L^p(\R^d)}\leq Ct,
\end{align}
and
\begin{align}
&2^{ns}\f\|u^{\rm{ap},3}_n-u^{\rm{ap},4}_n\g\|_{L^p(\R^d)}\leq C\int_0^t\tau2^{ns}\|(\nabla\Delta_nu_0)(x-\tau u_0(x))\|_{L^{p}(\R^d)}\dd\tau.
\end{align}
\end{proposition}
\begin{proof} By the Newton-Leibniz formula, one obtain from \eqref{YH1} that
\bbal
\|u(t)-u_0\|_{B^{s-1}_{p,\infty}(\R^d)}\leq&~ \int_0^t\|\mathcal{P}(u\cdot \nabla u)\|_{B^{s-1}_{p,\infty}(\R^d)}\dd\tau
\leq \int_0^t\left\|\mathcal{P}\left(u\cd\na u\right)\right\|_{\dot{B}^{s-1}_{p,\infty}(\R^d)\cap\dot{B}^0_{p,1}(\R^d)} \dd\tau\nonumber\\
\leq&~ C\int_0^t\left\|u\otimes u\right\|_{\dot{B}^{s}_{p,\infty}\cap\dot{B}^1_{p,1}(\R^d)} \dd\tau
\leq C\int_0^t\left\|u\otimes u\right\|_{B^{s}_{p,\infty}(\R^d)} \dd\tau\nonumber\\
\leq&~Ct\|u_0\|_{B^{s}_{p,\infty}(\R^d)}^2\leq Ct.
\end{align*}
Notice that (see Lemma \ref{lem:P})
\begin{align*}
\|\mathcal{Q}(u\cdot \nabla u-u_0\cdot \nabla u_0)\|_{B_{p, \infty}^{s-1}(\R^d)} & \leq \f\|\mathcal{Q}((u-u_0)\cdot \nabla u+u_0\cdot \nabla (u-u_0))\g\|_{B_{p, \infty}^{s-1}(\R^d)}\\
&\leq C\|u-u_0\|_{B_{p, \infty}^{s-1}(\R^d)}\f(\|u\|_{B_{p, \infty}^{s}(\R^d)}+\|u_0\|_{B_{p, \infty}^{s}(\R^d)}\g)\\
&\leq C\|u-u_0\|_{B_{p, \infty}^{s-1}(\R^d)},
\end{align*}
and applying Lemma \ref{zzgj}, we have
\bbal
\|u-u^{\rm{ap},1}\|_{B_{p,\infty}^{s-1}}\leq&~\int_0^t\|(u-u_0)\cd\na u^{\rm{ap},1}\|_{B_{p,\infty}^{s-1}(\R^d)}+\|\mathcal{Q}(u\cdot \nabla u-u_0\cdot \nabla u_0)\|_{B_{p,\infty}^{s-1}(\R^d)}\dd\tau\\
\leq&~ C\int_0^t\|u-u_0\|_{B_{p,\infty}^{s-1}(\R^d)}\dd\tau\leq Ct^2.
\end{align*}
Applying Lemma \ref{zzgj}, we have
\bbal
\|u^{\rm{ap},1}-u^{\rm{ap},2}\|_{B_{p,\infty}^{s}(\R^d)}\leq&~ \int_0^t\|\mathcal{Q}(u_0\cdot \nabla u_0)\|_{B_{p,\infty}^{s}(\R^d)}\dd\tau
\leq Ct.
\end{align*}
We denote $w:=u^{\rm{ap},3}_n-\Delta_nu^{\rm{ap},4}$. From \eqref{h1}, one has
\bbal
\bca
\pa_tw+u_0\cd\na w=tu_0\cdot\nabla u_0\cdot(\nabla\Delta_nu_0)(x-tu_0(x)),\\
w(0,x)=0.
\eca
\end{align*}
From the above, we obtain
\bbal
\|w\|_{L^p(\R^d)}\leq&~ \int_0^t\|\na u_0\|_{L^{\infty}(\R^d)}\|w\|_{L^p(\R^d)}\dd\tau\\
&\quad+\|u_0\|_{L^{\infty}(\R^d)}\|\nabla u_0\|_{L^{\infty}(\R^d)}\int_0^t\tau\|(\nabla\Delta_nu_0)(x-\tau u_0(x))\|_{L^{p}(\R^d)}\dd\tau\\
\leq&~C\int_0^t\|w\|_{L^p(\R^d)}\dd\tau+C\int_0^t\tau\|(\nabla\Delta_nu_0)(x-\tau u_0(x))\|_{L^{p}(\R^d)}\dd\tau.
\end{align*}
By Gronwall's inequality, we complete the proof of Proposition \ref{4pr1}.
\end{proof}
\subsection{Non-Continuous Dependence of Euler equations}
We should notice that, the higher-dimension case and the incompressibility condition makes the choice of initial data somewhat complicated. Motivated by the one dimension case (see Definition \ref{def}), we can construct technically the divergence-free vector field $u_0$.
Firstly, We need to introduce smooth, radial cut-off functions to localize the frequency region. Precisely, we define even, real-valued and non-negative functions $\widehat{\phi_1}\in \mathcal{C}^\infty_0(\mathbb{R})$ and $\widehat{\psi_1}\in \mathcal{C}^\infty_0(\mathbb{R})$ with values in $[0,1]$ which satisfy
\bbal
&\widehat{\phi_1}(\xi)=
\bca
1, \quad \mathrm{if}\  |\xi|\leq \frac{1}{4^d},\\
0, \quad \mathrm{if}\  |\xi|\geq \frac{1}{2^d},
\eca
\\
&\widehat{\psi_1}(\xi)=
\bca
1, \quad \mathrm{if}\ \frac12\frac{1}{\sqrt{d}}\leq |\xi|\leq \frac{5}{8}\frac{1}{\sqrt{d}},\\
0, \quad \mathrm{if}\  |\xi|\geq \frac{3}{4}\frac{1}{\sqrt{d}} \ \mathrm{or} \ |\xi|\leq  \frac38\frac{1}{\sqrt{d}}.
\eca
\end{align*}
It is easy to check that both $\phi_1(x)$ and $\psi_1(x)$ are real-valued and continuous functions on $\R$. Furthermore, it holds that $\|\phi_1\|_{L^\infty}=\phi_1(0)>0$ and $\|\psi_1\|_{L^\infty}=\psi_1(0)>0$.

\begin{definition}[Initial Data]\label{def-e} Let $s>1+d/p$ and $p\in[1,\infty]$. We define the divergence-free vector field $u_0$  as follows:
$$u_0(x):=\sum\limits^\infty_{j=3}2^{-j(s+1)}\nabla^\bot f_{j}(x),\quad x\in\R^d,$$
where we denote
\bbal
\nabla^\bot:=\bca
(\pa_2,\,-\pa_1), \quad  &d=2,\\
(\pa_2,\,-\pa_1,\,0,\ldots,\,0), \quad &d\geq3,
\eca
\end{align*}
and
\bbal
f_{j}(x):=\frac{8}{11}\sin \f(\frac{11}{8}2^{j}x_1\g)\prod_{i=1}^d\bar{\phi}(x_i)\quad\text{with}\quad
\bar{\phi}(x_i):=
\frac{\phi_1(x_i)}{\phi_1(0)}.
\end{align*}
\end{definition}
\begin{definition}[Perturbation of Initial Data]\label{def-e1} Let $u_0$ be given by Definition \ref{def-e}. Let $s>1+d/p$ and $p\in[1,\infty]$. We define the perturbation of $u_0$ as follows:
\bbal
u_0^{n}(x):=u_0(x)-\frac1n \nabla^\bot\bar{\Phi}(x),\quad x\in\R^d,
\end{align*}
where
\bbal
\bar{\Phi}(x):=\frac{1}{\psi^d_1(0)}(\pa_{x_2}^{-1}\psi_1)(x_2)\prod_{i\in\{1,3,\ldots,d\}}\psi_1(x_i).
\end{align*}
\end{definition}
\begin{remark}\label{hyl}
The anti-derivative $\pa_{x_2}^{-1}$ of $\psi_1$ can be defined by the Fourier multiplier with the symbol $\fr1\xi$, namely,
$$(\pa_{x}^{-1}\psi_1)(x):=\mathcal{F}^{-1}\f(-\mathrm{i}{\xi}^{-1}\widehat{\psi_1}(\xi)\g)(x).$$ The appearance of $\pa_{x_2}^{-1}\psi_1$ is to ensure that $\pa_{x_2}\Bar{\Phi}(0)=1$. In fact, $\pa_{x_2}\Bar{\Phi}(x)=\prod_{i=1}^d\frac{\psi_1(x_i)}{\psi_1(0)}.$
Thus one has $\f\|\widehat{\pa_{x_2}\Bar{\Phi}}\g\|_{L^1(\R^d)}=(2\pi)^d$ and
$$\pa_{x_2}\Bar{\Phi}(x)-\pa_{x_2}\Bar{\Phi}(0)=\f(\int_0^1\f(\nabla\pa_{x_2}\Bar{\Phi}\g)(\tau x)\dd\tau\g)x,$$
which furthermore gives that
$$\f|\pa_{x_2}\Bar{\Phi}(x)-\pa_{x_2}\Bar{\Phi}(0)\g|\leq C\f\|\nabla\pa_{x_2}\Bar{\Phi}\g\|_{L^\infty(\R^d)}|x|\leq C|x|.$$
This fact will be used later.
\end{remark}
\begin{lemma}\label{js2} Let $u_0$ and $u^n_0$ be given by Definitions \ref{def-e}-\ref{def-e1}. Assume that $(s,p)$ satisfies $s>1+d/p$ with $1\leq p\leq \infty$.
Then there exists some sufficiently large integer $M_0$ and constant $C_s>0$ independent of $M_0$  such that for arbitrarily  large integer $n\gg M_0$
\begin{align*}
&u^{(1)}_0(\mathbf{0})=0,\quad\Delta_n\Bar{\Phi}=0,\\
&\Delta_nu_0(x)=\Delta_nu^n_0(x)=2^{-n(s+1)}\nabla^\bot f_{n}(x), \\
&u_0\in W^{1,\infty}(\R^d)\cap B^s_{p,\infty}(\R^d) \quad \text{with}\quad \|u_0\|_{W^{1,\infty}(\R^d)\cap B^s_{p,\infty}(\R^d)}\leq C_s.
\end{align*}
\end{lemma}
\begin{proof} As in Lemma \ref{js1}, it is not difficult to complete the proof and we omit the details.
\end{proof}
Next, we need compute the difference between the approximate solution $u^{\rm{ap},4}_{n,1}$ of original solution and approximate solution $u^{\rm{ap},4}_{n,2}$ of perturbed solution.
\begin{proposition}\label{4pr2} Let $u_0$  be given by Definition \ref{def-e} and a perturbation sequence of initial data $u^n_0$ be given by Definition \ref{def-e1}. Setting
\bbal
u^{\rm{ap},4}_{n,1}=(\Delta_nu_0)(x-tu_0(x)) \quad\text{and}\quad u^{\rm{ap},4}_{n,2}=(\Delta_nu^n_0)(x-tu^{n}_0(x)).
\end{align*}
For small $t_n=\frac{8}{11}\pi n2^{-n}$, we have for some positive constant $C$ independent of $t_n$ and $n$
\begin{align*}
2^{ns}\f\|\f(u^{\rm{ap},4}_{n,1}-u_{n,2}^{\rm{ap},4}\g)(t_n)\g\|_{L^p(\R^d)}\geq 2^{-\frac{d}{p}M_0}\f(2^{-d-1}-C2^{-M_0}\g)-C2^{-n}-Ct_n.
\end{align*}
\end{proposition}
\begin{proof}
We notice that for some sufficiently large $n\in \mathbb{N}^+$
 \begin{align*}
 \Delta_{n}u_0=2^{-n(s+1)}
\f(
\pa_2f_n,\;
-\pa_1f_n,0,\cdots,0
\g).
    \end{align*}
From which and Lemma \ref{js2}, one has for $n\gg M_0$ ($M_0$ shall be fixed later)
\bbal
&\quad (\Delta_n
u_0)^{(2)}(x)=(\Delta_n
u^n_0)^{(2)}(x)\\
&=-2^{-ns}
\cos\f(\frac{11}{8}2^nx_1\g)\prod_{i=1}^d\bar{\phi}(x_i) -2^{-n(s+1)}
\frac{8}{11}\sin\f(\frac{11}{8}2^nx_1\g)\bar{\phi}'(x_1)\prod_{i=2}^d\bar{\phi}(x_i).
\end{align*}
Thus
\bbal
\f(2^{ns}u^{\rm{ap},4}_{n,2}\g)^{(2)}(t,x)
&=-\cos\f(\frac{11}{8}2^n\f(x_1-tu^{(1)}_0(x)-\frac{t}n\pa_2\bar{\Phi}(x)\g)\g)\prod_{i=1}^d\bar{\phi}(x_i)\\
&-\cos\f(\frac{11}{8}2^n\f(x_1-tu^{(1)}_0(x)-\frac{t}n\pa_2\bar{\Phi}(x)\g)\g)\f[\prod_{i=1}^d\bar{\phi}(x_i-t(u_0^n(x))^{(i)})-\prod_{i=1}^d\bar{\phi}(x_i)\g]\\
&-2^{-n}\frac{8}{11}\sin\f(\frac{11}{8}2^n(x_1-t(u_0^n(x))^{(1)})\g)\bar{\phi}'(x_1-t(u_0^n(x))^{(1)})\prod_{i=2}^d\bar{\phi}(x_i-t(u_0^n(x))^{(i)}),
\end{align*}
and
\bbal
\f(2^{ns}u^{\rm{ap},4}_{n,1}\g)^{(2)}(t,x)
&=-\cos\f(\frac{11}{8}2^n(x_1-tu^{(1)}_0(x))\g)\prod_{i=1}^d\bar{\phi}(x_i)\\
&-\cos\f(\frac{11}{8}2^n\f(x_1-tu^{(1)}_0(x)\g)\g)\f[\prod_{i=1}^d\bar{\phi}(x_i-tu^{(i)}_0(x))-\prod_{i=1}^d\bar{\phi}(x_i)\g]\\
&-2^{-n}\frac{8}{11}\sin\f(\frac{11}{8}2^n(x_1-tu^{(1)}_0(x))\g)\bar{\phi}'(x_1-tu^{(1)}_0(x))\prod_{i=2}^d\bar{\phi}(x_i-tu^{(i)}_0(x)).
\end{align*}
Letting $\frac{11}{8}2^{n}t_n=n\pi$, then one has
\bbal
&\quad 2^{ns}\f(u^{\rm{ap},4}_{n,2}-u_{n,1}^{\rm{ap},4}\g)^{(2)}(t_n,x)
=\mathcal{A}(x)
+\text{Remaining terms},
\end{align*}
where
$$\mathcal{A}(x):=\f[\cos\f(\frac{11}{8}2^n(x_1-t_nu^{(1)}_0(x))\g)-\cos\f(\frac{11}{8}2^n\f(x_1-t_nu^{(1)}_0(x)\g)-\pi\pa_2\bar{\Phi}(x)\g)\g]\prod_{i=1}^d\bar{\phi}(x_i).$$
By direct computations, we have
$$\f\|\text{Remaining terms}\g\|_{L^p(\R^d)}\leq C2^{-n}+Ct_n.$$

{\bf Case for $p=\infty$.}\; Obviously, $u^{(1)}_0(\mathbf{0})=0$ and $\pa_2\bar{\Phi}(\mathbf{0})=\bar{\phi}(0)=1$. Thus we have
\bbal
2^{ns}\f\|\f(u^{\rm{ap},4}_{n,2}-u^{\rm{ap},4}_{n,1}\g)^{(2)}(t_n,x)\g\|_{L^\infty(\R^d)}
&\geq \mathcal{A}(0)-C2^{-n}-Ct_n
\geq  2-C2^{-n}-Ct_n.
\end{align*}

{\bf Case for $1\leq p<\infty$.}\; Noticing that there exists some integer $M_0>0$ such that
\bbal
\forall x\in[0,2\pi2^{-M_0}], \qquad \phi_1(x)\geq \frac12\phi_1(0) \quad \Longleftrightarrow\quad \bar{\phi}(x)\geq\fr12,
\end{align*}
and using $\pa_2\bar{\Phi}(\mathbf{0})=1$ again, we obtain that
\bbal
&\quad 2^{ns}\f\|\f(u^{\rm{ap},4}_{n,2}-u^{\rm{ap},4}_{n,1}\g)^{(2)}(t_n)\g\|_{L^p(\R^d)}
\nonumber\\&\geq 2^{-d}\f\|\cos\f(\frac{11}{8}2^n\f(x_1-t_nu^{(1)}_0(x)\g)-\pi \pa_2\bar{\Phi}(x)\g)-\cos\f(\frac{11}{8}2^n(x_1-t_nu^{(1)}_0(x))\g)\g\|_{L^p([0,2\pi2^{-M_0}]^d)}\nonumber\\
&\quad-C2^{-n}-Ct_n\nonumber\\
&\geq2^{1-d}\f\|\cos\f(\frac{11}{8}2^n(x_1-t_nu^{(1)}_0(x))\g)\g\|_{L^p([0,2\pi2^{-M_0}]^d)}-C2^{-n}-Ct_n-C2^{-(1+\frac{d}{p})M_0},
\end{align*}
where we have used (see Remark \ref{hyl})
\bbal
&\quad\f\|\cos\f(\frac{11}{8}2^n\f(x_1-t_nu^{(1)}_0(x)\g)-\pi \pa_2\bar{\Phi}(x)\g)-\cos\f(\frac{11}{8}2^n\f(x_1-t_nu^{(1)}_0(x))\g)-\pi \pa_2\bar{\Phi}(0)\g)\g\|_{L^p([0,2\pi2^{-M_0}]^d)}\nonumber\\
&\leq \pi\f\|\pa_2\bar{\Phi}(x)-\pa_2\bar{\Phi}(0)\g\|_{L^p([0,2\pi2^{-M_0}]^d)}\leq
C2^{-(1+\frac{d}{p})M_0}.
\end{align*}
Thus by Lemma \ref{tl3}, one has
\bbal
2^{ns}\f\|\f(u^{\rm{ap},4}_{n,2}-u^{\rm{ap},4}_{n,1}\g)^{(2)}(t_n)\g\|_{L^p(\R^d)}\geq2^{-d-1-\frac{d}{p}M_0}-C2^{-n}-Ct_n-C2^{-(1+\frac{d}{p})M_0}.
\end{align*}
We complete the proof of Proposition \ref{4pr2}.
\end{proof}
{\bf Difference of initial data.}\; Obviously, we see that
\bbal
\|u^{n}_{0}-u_0\|_{B^s_{p,\infty}(\R^d)}= \frac1{n}\|\nabla^\bot\bar{\Phi}\|_{B^s_{p,\infty}(\R^d)}\leq \frac{C}n\to0,\quad \text{as}\;n\to\infty.
\end{align*}

{\bf Difference of solutions.}\; Using Proposition \ref{4pr1} and Proposition \ref{4pr2}, we deduce that for fixed large $M_0$
\bbal
\|u_n(t_n)-u(t_n)\|_{B^s_{p,\infty}(\R^d)}&\geq2^{ns}\|\Delta_n(u_n-u)(t_n)\|_{L^p(\R^d)}-Ct^2_n2^n-Ct_n-\frac{C}{n}\\
&\geq c2^{-\frac{d}{p}M_0}-Ct^2_n2^n-Ct_n-\frac{C}{n}.
\end{align*}
Letting $n\to\infty$, we complete the proof of Theorem \ref{th1}.

\section{Conclusion and Discussion }
In this paper, we consider the Cauchy problem for the Euler equations in the whole space and the torus, and show that the Euler equations are not
locally well-posed (in the sense of Hadamard) in $B^s_{p,\infty}$ and $C^{k,\alpha}$ due to the failure of the continuous dependence of solution map.
A remarkable novelty of the proof is the construction of an approximate solution to the Burgers equation. To the best of our knowledge, this new approximation technique is the first one addressing the ill-posedness issue on the transport equation. We should remark that the approach also possess other application. In particular, based on this approximate solution and modifying the initial data, we can prove that the solution map of the Euler equation cannot be continuous as a function of the time variable at time zero in $B^s_{p,\infty}(\R^d)$ (see Appendix).
Finally, it should be mention that, the method we used in this paper do not depend heavily on the form of explicit solutions of equations (it is general not easy to obtain especially for the whole space case) and can be applied some shallow water equations such as the Degasperis-Procesi, Camassa-Holm equation, Fornberg-Whitham equation and other transport type equations.
\section*{Appendix}
{\bf Proof of Corollary \ref{co2}}\,
We take the divergence-free vector field $u_0$ in the form
$$u_0(x)=2^{2s+3}(2^s-1)\sum\limits^\infty_{j=3}2^{-j(s+1)}\nabla^\bot \f[\cos \f(\frac{11}{8}2^{j}x_1\g)\frac{\pa_{x_2}^{-1}\psi_1(x_2)}{\psi_1(0)}\prod_{i\in\{1,3,\ldots,d\}}\frac{\phi_1(x_i)}{\phi_1(0)}\g],\;x\in \R^d.$$
Following the same method as that in Theorem \ref{th1}, one has
\bbal
\|u_n(t_n)-u_0\|_{B^s_{p,\infty}(\R^d)}
&\geq2^{ns}\f\|(\Delta_nu_0)^{(2)}(x-t_nu_0(x))-(\Delta_nu_0)^{(2)}(x)\g\|_{L^p(\R^d)}-Ct^2_n2^n-Ct_n
\\&\geq c\f\|\sin\f(\frac{11}{8}2^nx_1-\pi u^{(1)}_0(x)\g)-\sin\f(\frac{11}{8}2^nx_1\g)\g\|_{L^p([0,2\pi2^{-M_0}]^d)}-Ct^2_n2^n-Ct_n
\\
&\geq 2^{-\frac{d}{p}M_0}\f(c-C2^{-M_0}\g)-C2^{-n},
\end{align*}
which implies the result of Corollary \ref{co2}.

\section*{Declarations}
\noindent\textbf{Data Availability}\\
No data was used for the research described in the article.

\vspace*{1em}
\noindent\textbf{Conflict of interest}\\
The authors declare that they have no conflict of interest.
\vspace*{1em}

\noindent\textbf{Funding}\\
J. Li is supported by the National Natural Science Foundation of China (12161004), Innovative High end Talent Project in Ganpo Talent Program (gpyc20240069), Training Program for Academic and Technical Leaders of Major Disciplines in Ganpo Juncai Support Program (20232BCJ23009), Jiangxi Provincial Natural Science Foundation (20252BAC210004).

\end{document}